\documentclass[12pt,]{amsart}
\usepackage{amsmath}
\usepackage{amscd}
\usepackage{amssymb}
\usepackage{latexsym}

\newtheorem{theorem}{Theorem}[section]
\newtheorem{lemma}[theorem]{Lemma}
\newtheorem{corollary}[theorem]{Corollary}
\newtheorem{proposition}[theorem]{Proposition}

\newtheorem{remark}[theorem]{Remark}
\newtheorem{definition}[theorem]{Definition}

\newcommand{\bgl}{\begin{equation}}         %eine Gleichung mit Ziffer
\newcommand{\egl}{\end{equation}}
\newcommand{\bgln}{\begin{eqnarray}}        %mehrere Gleichungen mit %Ziffer
\newcommand{\egln}{\end{eqnarray}}
\newcommand{\bglnoz}{\begin{eqnarray*}}     %mehrere Gleichungen ohne %Ziffer
\newcommand{\eglnoz}{\end{eqnarray*}}
\newcommand{\btheo}{\begin{theorem}}
\newcommand{\etheo}{\end{theorem}}
\newcommand{\blemma}{\begin{lemma}}
\newcommand{\elemma}{\end{lemma}}
\newcommand{\bproof}{\begin{proof}}
\newcommand{\eproof}{\end{proof}}
\newcommand{\bbew}{\begin{beweis}}
\newcommand{\ebew}{\end{beweis}}
\newcommand{\bremark}{\begin{remark}\em}
\newcommand{\eremark}{\end{remark}}
\newcommand{\bdefin}{\begin{definition}}
\newcommand{\edefin}{\end{definition}}
\newcommand{\bprop}{\begin{proposition}}
\newcommand{\eprop}{\end{proposition}}
\newcommand{\bcor}{\begin{corollary}}
\newcommand{\ecor}{\end{corollary}}
\newcommand{\mn}{\par\medskip\noindent}

\newcommand{\cA}{\mathfrak A}

\newcommand{\cC}{\mathcal C}
\newcommand{\cD}{\mathcal D}
\newcommand{\cE}{\mathcal E}

\newcommand{\cI}{\mathcal I}

\newcommand{\cL}{\mathcal L}

\newcommand{\cP}{\mathcal P}

\newcommand{\cT}{\mathfrak T}

\newcommand{\lori}{\longrightarrow}

\newcommand{\ve}{\varepsilon}
\newcommand{\vp}{\varphi}
\newcommand{\Eins}{{\mathchoice{1\!\!1}{1\!\!1}{1\!\!1}{1\!\!1}}}

\def\SEMI{\mbox{$\times\kern-2pt\vrule height5pt width.6pt \kern3pt $}}

\newcommand{\Ker}{{\rm Ker\,}}
\newcommand{\Img}{{\rm Im\,}}
\newcommand{\Tr}{{\rm Tr\,}}

\newcommand{\tei}{\, | \,}

\newcommand{\Spec}{\mathrm{ Spec}\,}
\newcommand{\Rea}{\mathrm{ Re}}

\newcommand{\rxrx}{{R \rtimes R^\times}}
\newcommand{\rx}{R^{\times}}
\def\Cz{\mathbb{C}}

\def\Nz{\mathbb{N}}

\def\Qz{\mathbb{Q}}
\def\Rz{\mathbb{R}}
\def\Zz{\mathbb{Z}}

\def\Tz{\mathbb{T}}
\newcommand{\tr}{\operatorname{tr}}      % trace

\newcommand{\abs}[1]{\lvert#1\rvert}     % absolute value
\newcommand{\defeq}{\mathrel{:=}}     % equal by definition

\newcommand{\inv}{^{-1}}
\def\tdj{\tilde\delta_{J_\gamma}}
\def\tdi{\tilde\delta_I}
\newcommand\af{{\mathbb{A}_f}}
\newcommand\afs{{\mathbb{A}^*_f}}

\def\rhat{{\hat R}}
\def\oma{\Omega_\af}
\def\omr{\Omega_\rhat}
\newcommand{\kxkx}{{K\rtimes K^*}}
\newcommand\chf{{\Eins}}
\def\nmi{\underline{\cI_k}}

\begin{document}
\title[C*-algebras of Toeplitz type from number fields]{$C^*$-algebras of
Toeplitz type associated with algebraic
number fields}
\date{\today}
\author[J. Cuntz]{Joachim Cuntz$^1$}
\author[C. Deninger]{Christopher Deninger$^2$}
\author[M. Laca]{Marcelo Laca$^3$}
\address{Joachim Cuntz, Mathematisches Institut, Einsteinstr.62, 48149
M\"unster, Germany\\Christopher Deninger, Mathematisches Institut,
Einsteinstr.62, 48149 M\"unster, Germany\\Marcelo Laca, Mathematics
and Statistics, University of Victoria PO BOX 3060 STN CSC,
Victoria, B.C. Canada V8W 3R4} \email{cuntz@uni-muenster.de,
deninger@uni-muenster.de, laca@uvic.ca}
\thanks{{$^1$}Research supported by DFG through
CRC 878 and by ERC through AdG 267079,{$^2$}Research supported by
DFG through CRC 878, {$^3$}Research supported by NSERC and PIMS}
\begin{abstract}\noindent
We associate with the ring $R$ of algebraic integers in a number
field a C*-algebra $\cT[R]$. It is an extension of the ring
C*-algebra $\cA[R]$ studied previously by the first named author in
collaboration with X.Li. In contrast to $\cA[R]$, it is functorial
under homomorphisms of rings. It can also be defined using the left
regular representation of the $ax+b$-semigroup $R\rtimes R^\times$
on $\ell^2 (R\rtimes R^\times )$.

The algebra $\cT[R]$ carries a natural one-parameter automorphism
group $(\sigma_t)_{t\in\Rz}$. We determine its KMS-structure. The
technical difficulties that we encounter are due to the presence of
the class group in the case where $R$ is not a principal ideal
domain. In that case, for a fixed large inverse temperature, the
simplex of KMS-states splits over the class group. The ``partition
functions'' are partial Dedekind $\zeta$-functions. We prove a
result characterizing the asymptotic behavior of quotients of such
partial $\zeta$-functions, which we then use to show uniqueness of
the $\beta$-KMS state for each inverse temperature $\beta\in(1,2]$.
\end{abstract}\maketitle

\section{Introduction}
Let $R$ be the ring of algebraic integers in a number field $K$, let
$R^\times=R\backslash \{0\}$ be its multiplicative semigroup and
$R\rtimes R^\times$ its $ax+b$-semigroup. In the present paper we
study the C*-algebra generated by the left regular representation of
the semigroup $R\rtimes R^\times$ on the Hilbert space $\ell^2
(R\rtimes R^\times )$, and its KMS-structure for a natural
one-parameter automorphism group. In the first part of the paper we
analyze the structural properties of the C*-algebra. We show that it
can be described as a universal C*-algebra defined by generators and
relations. Since the left regular C*-algebra of a semigroup is often
called its Toeplitz algebra we denote this universal algebra by
$\cT[R]$. The relations are closely related to those characterizing
the ring C*-algebra $\cA[R]$ studied in \cite{Cun}, \cite{CuLi},
\cite{CuLi2}. This is corresponds to the fact that $\cA[R]$ is
generated by the natural representation of $R\rtimes R^\times$ on
$\ell^2 (R)$ rather than on $\ell^2 (R\rtimes R^\times)$. Recall
that the generators for $\cA[R]$ are unitaries $u^x,\, x\in R$ and
isometries $s_a,\, a\in R^\times$ satisfying the following
relations:

\begin{itemize}
  \item[(a)] The $u^x$ and the $s_a$ define representations of the
  additive group $R$ and of the multiplicative semigroup $R^\times$,
  respectively,
  (i.e. $u^xu^y=u^{x+y}$ and $s_as_b=s_{ab}$) and moreover we
  require the relation $s_a u^x=u^{ax}s_a$ for all $x\in R,\,a\in
  R^\times$ (i.e the $u^x$ and $s_a$ together give a representation
  of the $ax+b$-semigroup $R\rtimes R^\times$).
  \item[(b)] For each $a\in R^\times$ one has $\sum_{x\in
  R/aR}u^xs_as^*_au^{-x}=1$.
\end{itemize}

This algebra was shown to be purely infinite and simple in
\cite{Cun}, \cite{CuLi}. It has different representations in terms
of crossed products for actions on spaces of finite or infinite
adeles for $K$, \cite{CuLi2}.

Now, to obtain a presentation of $\cT[R]$ we essentially have to
relax, in this presentation of $\cA[R]$, condition (b) to the weaker
condition $\sum_{x\in R/aR}u^xs_as^*_au^{-x}\leq 1$. This
modification in the relations is sufficient to characterize the
algebra $\cT[R]$ in the case where $R$ is a principal ideal domain.
We are however especially interested precisely in the the situation
where this is not the case, i.e. where the number field $K$ has
non-trivial class group. To treat this case adequately we have to
impose certain conditions on the range projections of the isometries
$s_a$. The most efficient way to formalize these conditions is to
use projections associated with ideals in $R$ as additional
generators and to describe their relations. We mention that a
description of the C*-algebra generated by the left regular
representation of a cancellative semigroup by analogous generators
and relations had been discussed before also by X. Li, \cite{Li},
Appendix A2, see also \cite{peebles}, Chapter 4 for a specific
example.

An important role in our analysis of $\cT[R]$ is played by a
canonical maximal commutative subalgebra. Its Gelfand spectrum $Y_R$
can be understood as a completion, for a natural metric, of the
disjoint union $\bigsqcup R/I$ over all non-zero ideals $I$ of $R$.
It contains the profinite completion $\hat{R}$ of $R$ (which is the
spectrum of the analogous commutative subalgebra of $\cA[R]$). It is
important to note that the algebra $\cT[R]$ is functorial for
homomorphisms between rings while $\cA[R]$ is not. This is reflected
in the striking fact that the construction $R\mapsto Y_R$ is
contravariant under ring homomorphisms rather than covariant as one
might expect. An inclusion of rings $R\subset S$ induces a
surjective map $Y_S\to Y_R$.  The same holds for the locally compact
version of $Y_R$ (corresponding to a natural stabilization of
$\cT[R]$) which plays the role of the locally compact space of
finite adeles.

Especially important for us is a natural one-parameter group
$(\sigma_t)_{t\in\Rz}$ of automorphisms of $\cT[R]$. It is closely
related to Bost-Connes systems \cite{BoCo} and to Dedekind
$\zeta$-functions. In special cases it had been considered before
in \cite{Cun}, \cite{LaRa}.

The Toeplitz algebra for the semigroup $\Nz\rtimes \Nz^\times$ -
which is very closely related to the Toeplitz algebra $\cT[\Zz]$ for
the ring $\Zz$ in the sense of the present paper - has been analyzed
in \cite{LaRa}. In particular it was found in that paper that the
canonical one-parameter automorphism group on this algebra has an
intriguing KMS-structure. There is a phase transition at $\beta =2$
with a spontaneous symmetry breaking. In the range $1\leq\beta\leq
2$ there is a unique KMS-state while for $\beta>2$ there is a family
of KMS-states labeled by the probability measures on the circle and
with partition function the Riemann $\zeta$-function.

It turns out that, for our Toeplitz algebra, the KMS-structure is
similar, but quite a bit more intricate. We show in Theorem 6.7 that
for $\beta$ in the range $1\leq\beta\leq 2$ (with $\beta=1$ playing
a special role) there is a unique KMS state. The essential new
feature which is also the source of the main technical difficulties
in this paper is the presence of the class group, in the case where
$R$ is not a principal ideal domain. Our proof for the uniqueness of
the KMS-state requires a delicate estimate of the asymptotics of
partial Dedekind $\zeta$-functions for different ideal classes, see
Theorem \ref{den}. This theorem seems to be new and of independent
interest. We include the proof in the appendix.

For $\beta >2$ we obtain a splitting of the KMS states over the
class group $\Gamma$ for the number field $K$. The KMS states for
each $\beta$ in this range are labeled by the elements $\gamma\in
\Gamma$, but moreover also by traces on a crossed product
$\cC(\Tz^n)\rtimes R^*$ ($n$ being the degree of our field
extension) by an action (which depends on $\gamma$) of the group
$R^*$ of units of $R$. For a precise statement see Theorem 7.3. The
partition functions are the partial Dedekind $\zeta$-functions
$\zeta_\gamma$ associated with the ideal classes $\gamma$ for $K$.

In section \ref{gs} we determine the ground states. We find a
situation which is similar to the one for the KMS states in the
range $\beta>2$. The ground states are labeled by the states of a
certain subalgebra of $\cT[R]$.

We mention that our methods also immediately yield the KMS-structure
of the much simpler, but in the case of a non-trivial class group
still interesting, C*-dynamical system that one obtains from the
Toeplitz algebra of the multiplicative semigroup $R^\times$ (i.e.
the C*-algebra generated by the left regular action of this
semigroup on $\ell^2(R^\times)$) with the analogous one-parameter
automorphism group, see Remark \ref{Rtimes}.

When we restrict to the case of a trivial class group, all our
arguments become very simple indeed and can be used to get a simpler
approach to the results in \cite{LaRa}.

The presentation of $\cT[R]$ in terms of generators and relations
and the functoriality from section 3 had been obtained and announced
by the first named author before the present paper took shape. These
two results have since been generalized to more general semigroups
by Xin Li, \cite{LiSG}. The first named author is indebted to Peter
Schneider for very helpful comments.

After this paper was circulated, S. Neshveyev informed us that,
using the crossed product description of $\cT[R]$ in section 5 and
the methods developed in \cite{LaNeTr}, the KMS-structure on
$(\cT[R], (\sigma_t))$ could be linked to that of a Bost-Connes
system. The KMS-structure of this Bost-Connes system in turn was
determined in \cite{LaLaNe}. Together, this would give a basis for
an alternative approach to our results on KMS-states in sections 6
and 7.

We include a brief list of notations at the end of the appendix.

\section{The Toeplitz algebra for the $ax+b$-semigroup over $R$}\label{basic}

Let $R$ be the ring of algebraic integers in the number field $K$.
The $ax+b$-semigroup for $R$ is the semidirect product $R\rtimes
R^\times$ of the additive group $R$ and the multiplicative semigroup
$R^\times= R\setminus \{0\}$ of $R$. We can define the Toeplitz
algebra for the semigroup $R\rtimes R^\times$ as the C*-algebra
generated by the left regular representation of $R\rtimes R^\times$
on $\ell^2(R\rtimes R^\times)$. We set out to describe this
C*-algebra abstractly as a C*-algebra given by generators and
relations.

\bdefin We define the C*-algebra $\cT[R]$ as the universal
C*-algebra generated by elements $u^x, x\in R$, $s_a,\, a \in
R^\times$, $e_I$, $I$ a non-zero ideal in $R$, with the following
relations
\begin{itemize}
  \item[Ta:] The $u^x$ are unitary and satisfy $u^xu^y=u^{x+y}$, the
  $s_a$ are isometries and satisfy $s_as_b=s_{ab}$. Moreover we
  require the relation $s_a u^x=u^{ax}s_a$ for all $x\in R,\,a\in
  R^\times$.
  \item[Tb:] The $e_I$ are projections and satisfy $e_{I\cap J}=e_Ie_J$, $e_R=1$.
  \item[Tc:] We have $s_ae_Is_a^* = e_{aI}$.
  \item[Td:] For $x\in I$ one has $u^xe_I = e_Iu^x$, for $x\notin I$
  one has $e_Iu^xe_I=0$.
\end{itemize}\edefin
The first condition Ta simply means that the $u^x$ and $s_a$ define
a representation of the semigroup $R\rtimes R^\times$. We will see
below that $\cT[R]$ is actually isomorphic to the Toeplitz algebra
for the $ax+b$-semigroup $R\rtimes R^\times$, see Corollary \ref{416}.

In the following, ideals in $R$ will always be understood to be
non-zero.

\bremark In the case where $R$ is a principal ideal domain, the
axioms can be reduced considerably. In fact, in that case, the
projections $e_I$ are not needed to describe $\cT[R]$ by generators
and relations (they are all of the form $s_as^*_a$) and conditions
Tb, Tc and Td can be replaced by the single very simple condition

$$\sum_{x\in R/aR}u^xs_as^*_au^{-x}\leq 1$$ Note that this inequality
is a consequence of Tc and Td. In fact, by Tc one has
$e_{aR}=s_as_a^*$ and by Td the projections $u^xe_{aR}u^{-x}$, $x\in
R/aR$ are pairwise orthogonal.\eremark

\bremark\label{rem1} (a) The elements $s_a$ and $s_b^*$ commute if
and only if $a$ and $b$ are relatively prime (i.e. $aR+bR=R$).
(Proof: $s_b^*s_a = s_as_b^*$ iff $s_bs_b^*s_as_a^*=
s_bs_as_b^*s_a^*$ iff $e_{aR}e_{bR} = e_{abR}$. Then use the fact, established below using explicit representations of $\cT[R]$, that $e_I=e_J\Rightarrow I=J$)\mn (b) From
condition Td it follows that $e_Iu^xe_J=0$ if $x\notin I+J$ and that
$e_Iu^xe_J=u^{x_1}e_{I\cap J}u^{x_2}$ if there are $x_1\in I$ and
$x_2\in J$ such that $x=x_1+x_2$.\eremark

Let us derive a few more consequences from the axioms Ta - Td. From
the projections $e_I$ we can form associated projections. For each
ideal $I$ in $R$ set
$$
f_I = \sum_{x\in R/I} u^x e_I u^{-x}
$$
(note that $u^x e_I u^{-x}$ is well defined for $x\in R/I$, since
$u^{x+i} e_I u^{-x-i}=u^x e_I u^{-x}$ for $i\in I$, and that the
$u^x e_I u^{-x}$ are pairwise orthogonal for different $x\in R/I$).
For each prime ideal $P$ and $n\in\Nz$ set
$$
\varepsilon_{P^n} = f_{P^{n-1}} - f_{P^{n}}
$$
Moreover, for an ideal $I=P_1^{k_1}P_2^{k_2}\ldots P_n^{k_n}$ with
$P_1,P_2,\cdots P_n$ distinct primes, set
$$
\varepsilon_I = \varepsilon_{P_1^{k_1}}\varepsilon_{P_2^{k_2}}\cdots
\varepsilon_{P_n^{k_n}}
$$
\blemma\label{prop} The $e_I,f_I,\ve_I$ have the following
properties:
\begin{itemize}
  \item[(a)] For any two ideals $I$ and $J$ in $R$ one has
  $$e_If_J =\sum_{x\in I/(I\cap J)}u^xe_{I\cap J}u^{-x} \qquad
  f_If_J =\sum_{x\in (I+J)/(I\cap J)}u^xe_{I\cap J}u^{-x}$$
  \item[(b)] If $I$ and $J$ are relatively prime, then
  $f_If_J=f_{IJ}$. If $I\subset J$, then $f_If_J =f_I$.
  \item[(c)] If $I$ and $J$ are relatively prime, then
  $\ve_I\ve_J=\ve_{IJ}$. If $I$ and $J$ have a common prime divisor but occuring with different multiplicities, then $\ve_I\ve_J=0$.
  \item[(d)] The family of projections
  $\{e_I,f_I,\ve_I \big|\,I \,\mbox{an ideal in} \,R\}$ is commutative.
  \item[(e)] $u^x f_Iu^{-x}=f_I$ and $u^x
  \ve_Iu^{-x}=\ve_I$ for all $x\in R$.
\end{itemize}
\elemma \bproof (a) Obvious from Remark \ref{rem1}.

(b) is a special case of the formula under (a).

(c) follows from the definition together with the fact that
$\ve_{P^n}\ve_{P^m}=0$ for a prime ideal $P$ and $n\neq m$.

(d) It follows from (a) that the $e_I$ and $f_I$ form a commutative
family. However the $\ve_I$ are defined as products of differences
of certain $f_J$.

(e) follows directly from the definition.\eproof

\section{Functoriality of $\cT[R]$ for injective homomorphisms of
rings}

We assume that we have an inclusion $R\subset S$ of rings of
algebraic integers. We are going to show that this induces an
(injective) homomorphism $\kappa: \cT[R]\to\cT[S]$. Denote by
$s_a,u^x,e_I$ the generators of $\cT[R]$ and by
$\bar{s}_a,\bar{u}^x,\bar{e}_I$ the generators of $\cT[S]$.

The homomorphism $\kappa$ will map $s_a$ to $\bar{s}_a$, $u^x$ to $
\bar{u}^x$ and it is clear that this respects the relations Ta. With
an ideal $I$ in $R$ we associate the ideal $IS$ in $S$ and we define
$\kappa (e_I)=\bar{e}_{IS}$. It is then clear that relation Tc is
also respected. The fact that Tb and Td are respected follows from
the following elementary (and well-known) lemma.

\blemma In the situation above one has for ideals $I,J$ in R:
\begin{itemize}
  \item[(a)] $IS\cap R=I$
  \item[(b)] $IS\cap JS=(I\cap J)S$
\end{itemize}\elemma
\bproof Both statements can be proved in an elementary way using the
unique decomposition of $I$ and $J$ into prime ideals in $R$, cf.
\cite{Neukirch}, p. 45 and p. 52, Exercise 1. The statements also
follow from the fact that $S$ is a flat (even projective) module
over $R$, see \cite{BourbakiCA}, Chap. I {\S}2.6 Prop. 6 and
Corollary. \eproof

Summarizing, we obtain

\bprop Let $R$ and $S$ be the rings of algebraic integers in the
number fields $K$ and $L$, respectively. Then any injective
homomorphism $\alpha :R\to S$ induces naturally a homomorphism
$\cT[R]\to \cT[S]$.\eprop

It follows from Theorem \ref{inj}
 below that this homomorphism is also injective.
\section{The canonical commutative subalgebra}\label{D}
We denote by $\bar{\cD}$ the C*-subalgebra of $\cT[R]$ generated by
all projections of the form $u^xe_Iu^{-x}$, $x\in R$, $I$ a non-zero
ideal in $R$. It follows from Remark \ref{rem1} (b) that this
algebra is commutative. In fact, the elements $e^x_I:= u^x e_I
u^{-x}$  satisfy
 \begin{equation*}
 e^x_I e^y_J = \begin{cases} 0 & \text{ if } (x+I) \cap (y+J) = \emptyset\\
e^z_{I\cap J} &  \text{ if } z \in (x+I) \cap (y+J)
\end{cases}
  \end{equation*}
and thus linearly span a dense $*$-subalgebra of $\bar \cD$. The
algebra $\bar{\cD}$ also obviously contains the elements of the form
$\ve_I$ defined above.

\blemma\label{elem} \hspace*{-3em}\begin{itemize}
     \item[(a)] If $d\in\bar{\cD}$, then $s_ads_a^*$ and $u^xdu^{-x}$ are
     in $\bar{\cD}$ for all $a\in R^\times,\,x\in R$.
     \item[(b)] The set of linear combinations of elements of the
     form $s_a^* du^xs_b$ with $a,b\in R^\times,\;x\in R,\,d\in \bar{\cD}$
     is a dense $\star$-subalgebra in $\cT[R]$.
\end{itemize} \elemma

\bproof (a) This follows from the definition and conditions Ta -
Td.

(b) The set of elements of the form $s_a^* du^xs_b$ contains the
generators and, by (a),  is invariant under adjoints and multiplication from the
left or from the right by elements $s_c,s_c^*, u^y, e_I$ for $c\in
R^\times,\,y\in R$, $I$ an ideal in $R$ (the invariance under multiplication by $s_c^*$
on the right follows from the identity
$s_bs^*_c=s^*_cs_cs_bs^*_c=s^*_cs_bs_cs^*_c=s^*_ce_{bcR}s_b$).\eproof

Let $P$ be a prime ideal in $R$. We denote by $\bar{\cD}_P$ the
C*-subalgebra of $\bar{\cD}$ generated by all projections of the
form $u^xe_{P^n}u^{-x}$ with $x\in R$, $n=0,1,2,\ldots$. The
$\ve_{P^n}$ define pairwise orthogonal projections in $\bar{\cD}_P$.
We define projections in $\bar{\cD}_P$ by

$$\delta_{P^n}^x = u^x e_{P^n} \ve_{P^{n+1}}u^{-x}= u^x e_{P^n}
u^{-x}\ve_{P^{n+1}}, \; x\in R/P^{n}
$$
They are pairwise orthogonal since $\ve_{P^n}$ and $\ve_{P^m}$ are
orthogonal for $n\neq m$ and since the $u^xe_{P^n}u^{-x}$ are
pairwise orthogonal. In our definition we allow for $n=0$ so that
$\delta^x_{P^0}= \ve_P$. Note that, by Lemma \ref{min} below, the
$\delta^x_{P^0}= \ve_P$ are all non-zero.

\blemma\label{delta} One has
$$
\delta_{P^n}^0 = e_{P^n} - \sum_{x\in
P^n/P^{n+1}}u^xe_{P^{n+1}}u^{-x} \qquad\mbox{and}\qquad
\ve_{P^{n+1}} = \sum_{x\in R/P^{n}} \delta_{P^n}^x
$$
\elemma \bproof \bglnoz\delta^0_{P^n} = e_{P^n}\ve_{P^{n+1}} =
e_{P^n}\left(\sum_{x\in P^{n-1}/P^n} u^x e_{P^n}u^{-x} - \sum_{y\in
P^n/P^{n+1}} u^y e_{P^{n+1}}u^{-y} \right) \\[4pt] = e_{P^n} - \sum_{y\in
P^n/P^{n+1}} u^y e_{P^{n+1}}u^{-y} \eglnoz The second identity is
obvious from either formula for the $\delta^x_{P^n}$. \eproof

\blemma\label{DP} The algebra $\ve_{P^n}\bar{\cD}_P$ is
finite-dimensional and isomorphic to $\cC(R/P^{n-1})$. For each
$\{x\}$ in $\cC(R/P^{n-1})$, the projection $\,\delta_{P^{n-1}}^x =
u^xe_{P^{n-1}}\ve_{P^n}u^{-x},\,x\in R/P^{n-1}$ is minimal in this
algebra and corresponds to the characteristic function of $\{x\}$.
The isomorphism $\ve_{P^n}\bar{\cD}_P\cong \cC(R/P^{n-1})$ is
compatible with the natural action of the additive group $R$ on
these two algebras.

For each $k\leq n$ we have $$\ve_{P^{n+1}}e_{P^k}=\sum_{x\in
P^k/P^{n}}\delta^x_{P^n}$$

Let $G_k$ denote the (finite-dimensional) C*-subalgebra of
$\bar{\cD}_P$ generated by the projections $1,u^xe_{P^i}u^{-x}$,
$i=1,\ldots ,k,\,x\in R/P^i$. The map $$G_k\backepsilon z\mapsto
(z\ve_P,z\ve_{P^2},\ldots,z\ve_{P^{k+1}})$$ defines an isomorphism
$G_k\to \bigoplus_{n\leq\, k+1}\ve_{P^n}\bar{\cD}_P$. \elemma

\bproof For $k\geq n$, since $e_{P^k}\leq f_{P^n}$, we have
$e_{P^k}\ve_{P^n}=0$ and, since $u^x$ commutes with $\ve_{P^n}$,
also $u^xe_{P^k}u^{-x}\ve_{P^n}=0$ for such $k$.

On the other hand if $k\leq n-1$, then $e_{P^k}u^xe_{P^{n-1}}=0$ for
$x\notin P^k$ and $e_{P^k}u^xe_{P^{n-1}}=u^xe_{P^{n-1}}$ for $x\in
P^k$. Applying this to the product $e_{P^k}\delta^x_{P^n}=e_{P^k}u^x
e_{P^n}u^{-x}\ve_{P^{n+1}}$ we see that this expression vanishes for
$x\notin P^k$ and equals $\delta^x_{P^n}$ for $x\in P^k$.

The last assertion then is an immediate consequence. \eproof

Denote by $\cD_P$ the ideal in $\bar{\cD}_P$ generated by the
$\ve_{P^n}$. Lemma \ref{DP} shows that $\cD_P\cong \bigoplus
\cC(R/P^n)$. Since the union of the subalgebras $G_k$ is dense in
$\bar{\cD}_P$, the last statement in this lemma also shows that
$\cD_P$ is an essential ideal in $\bar{\cD}_P$.

Let $\iota: \cC(R/P^{n})\to \cC(R/P^{m})$ denote the homomorphism
induced by the quotient map $R/P^m\to R/P^{n}$ for $m>n$.

\blemma The C*-algebra $\bar{\cD}_P$ is isomorphic to the subalgebra
of the infinite product
$$ \prod_{n=0}^\infty \cC(R/P^n )$$
given by the ``Cauchy sequences'' $(d_n)$ (by this we mean that for
each $\ve
>0$ there is $N>0$ such that $\|\iota (d_n) - d_m\| <\ve$ for all
$n,m$ such that $N\leq n \leq m$).\elemma

\bproof The map $\bar{\cD}_P\ni z\mapsto (\ve_{P^n}z)\in
\prod_{n=0}^\infty \cC(R/P^n )$ is injective since $\cD_P$ is
essential. Each element of the form $u^xe_{P^k}u^{-x}$ is mapped,
according to Lemma \ref{DP} to the sequence $$(0,\ldots
,0,\delta^x_{P^k},\iota (\delta^x_{P^k}),
\iota^2(\delta^x_{P^k}),\ldots)$$ Thus the images of these
projections generate, together with the images of the $\delta^x_{P^k}$, the algebra of all ``Cauchy-sequences''.\eproof

\blemma\label{specD} There is an exact sequence
$$ 0 \to \cD_P \to \bar{\cD}_P \to \cC \to 0$$
where the Gelfand spectrum $\Spec \cD_P$ equals $\bigsqcup R/P^n$
and the Gelfand spectrum of the C*-algebra $\cC$ is the $P$-adic
completion $R_P=\mathop{\lim}\limits_{{
{\textstyle\longleftarrow}\atop{\scriptstyle n} }}R/P^n$ of
$R$.\elemma

\bproof The isomorphism $\cD_P\cong \bigoplus \cC(R/P^n)$ shows that
$\Spec \cD_P=\bigsqcup R/P^n$. In the quotient $\cC
=\bar{\cD}_P/\cD_P$ the images $(u^xe_{P^n}u^{-x})\,\breve{}$ of the
projections $u^xe_{P^n}u^{-x}$ satisfy the relation

$$\breve{e}_{P^n} = \sum_{x\in
P^n/P^{m}}(u^xe_{P^{m}}u^{-x})\,\breve{}$$

for $m\geq n$ (see Lemma \ref{delta}). Since $\cC$ is generated by
these images, $\cC=\mathop{\lim}\limits_{{
{\textstyle\longrightarrow}\atop{\scriptstyle n} }}\cC(R/P^n)$ and this
proves the second claim.\eproof

Now let $P_1, P_2, \ldots $ be an enumeration of the prime ideals in
$R$ (say ordered by increasing norm $|R/P_i|$) and, for each $n$, let $\cI_n$ be the set of
ideals of the form $I =P_1^{k_1}P_2^{k_2}\cdots P_n^{k_n}$ with all
$k_i \geq 0$. We write $\cD_n = \cD_{P_1}\cD_{P_2}\cdots \cD_{P_n}$
and $\bar{\cD}_n = \bar{\cD}_{P_1}\bar{\cD}_{P_2}\cdots
\bar{\cD}_{P_n}$. The $\bar{\cD}_{P_n}$ all commute and $\bar{\cD}$
obviously is the inductive limit of the $\bar{\cD}_n$.

We now use a natural representation of $\cT[R]$ on the following
Hilbert space $H_R$:

$$ H_R = \bigoplus_{I \,\textrm{ideal in}\, R} \ell^2 (R/I) $$

Note that $H_R$ is isomorphic to the infinite tensor product
$\bigotimes_P H_P$ where the tensor product is taken over all prime
ideals $P$ in $R$ and $H_P=\bigoplus \ell^2 (R/P^n)$ with ``vacuum
vector'' the standard unit vector in the one-dimensional space
$\ell^2 (R/R)$.\mn

$\cT[R]$ acts on $H_R$ in the following way:
\begin{itemize}
  \item The unitaries $u^x$, $x\in R$ act componentwise on $\ell^2(R/I)$ in
  the natural way.
  \item The isometries $s_a$ act through the composition:
  $\ell^2(R/I)\cong \ell^2 (aR/aI) \hookrightarrow \ell^2 (R/aI)$.
  \item The projection $e_J$ is represented by the orthogonal
  projection onto the subspace
  $ H = \bigoplus_{I\subset J} \ell^2 (J/I) $ of $H$.
\end{itemize}
It is easy to check that this assignment respects the relations
between the generators and thus defines a representation $\mu$ of
$\cT[R]$. One has

\blemma\label{min} Let $I=P_1^{k_1}P_2^{k_2}\cdots P_n^{k_n}$ with
all $k_i\geq 1$, and $x_1,x_2,\ldots, x_n \in R$. Then $\mu
(\delta^0_{P_1^{k_1}})$ acts on the subspace $\ell^2(R/I)$ of $H_R$
as the orthogonal projection onto the subspace
$\ell^2(P_1^{k_1}/I)$. Thus
$\mu(\delta^0_{P_1^{k_1}}\delta^0_{P_2^{k_2}}\cdots
\delta^0_{P_n^{k_n}})$ acts on this subspace as the orthogonal
projection onto the one-dimensional subspace $\ell^2(I/I)$ and $\mu(
\delta^{x_1}_{P_1^{k_1}}\delta^{x_2}_{P_2^{k_2}}\cdots
\delta^{x_n}_{P_n^{k_n}})$ acts as the orthogonal projection onto
the one-dimensional subspace $\ell^2((I+z)/I)$ where $z$ is the
unique element in $\bigcap_i (P_i^{k_i}+x_i)/I$.\elemma

\bproof This follows from the definition of $\mu (e_{P_1^{k_1}})$
and the fact that $\ve_{P_1^{k_1+1}}=1$ on $\ell^2(R/I)$ (recall
that, by definition
$\delta^0_{P_1^{k_1}}=e_{P_1^{k_1}}\ve_{P_1^{k_1+1}}$).\eproof

\blemma\label{prodel} For an ideal $I=P_1^{k_1}P_2^{k_2}\cdots P_n^{k_n}$ in $\cI_n$
and $x\in R/I$ consider the projection
$$\delta_{I,n}^x = u^x\delta^0_{P_1^{k_1}}\delta^0_{P_2^{k_2}}\cdots
\delta^0_{P_n^{k_n}}u^{-x}, \; x\in R/I
$$
These projections are non-zero according to Lemma \ref{min}. (Note
also that our definition of $\delta_{I,n}^x$ depends on the fact
that we consider $I$ as an element of $\cI_n$!) Then
$$
\delta_{I,n}^x=u^x e_I \ve_{IP_1P_2\cdots P_n}
u^{-x}\qquad\mbox{and}\qquad\ve_{IP_1P_2\cdots P_n} = \sum_{x\in
R/I} \delta_{I,n}^x
$$
\elemma

\bproof The identity $\delta_{I,n}^0=e_I \ve_{IP_1P_2\cdots P_n}$
follows from the equations $e_I=e_{P_1^{k_1}}e_{P_2^{k_2}}\ldots
e_{P_n^{k_n}}$ and $\ve_{IP_1P_2\cdots
P_n}=\ve_{P_1^{k_1+1}}\ve_{P_2^{k_2+1}}\ldots \ve_{P_n^{k_n+1}}$
(see Lemma \ref{prop}). The second identity follows from the first
one in combination with the corresponding identity in Lemma
\ref{delta}.\eproof

We have now shown that $\cD_n=\cD_{P_1}\cD_{P_2}\ldots\cD_{P_n}$ is
isomorphic to the tensor product $\bigotimes_{1\leq i\leq
n}\cD_{P_i}$ with minimal projections the $\delta_{I,n}^x$, $I\in
\cI_n$. Thus $\cD_n\cong \bigoplus_{I\in \cI_n}\cC (R/I)$ and the
spectrum of $\cD_n$ is $\bigsqcup_{I\in \cI_n} R/I$ (this is the
cartesian product of the spectra $\bigsqcup_{k\geq 0} R/P_i^k$ of
the $\cD_{P_i}$).

\bcor $\cD_n$ is an essential ideal in $\bar{\cD}_n$. $\bar{\cD}_n$
is isomorphic to $\bar{\cD}_{P_1}\otimes
\bar{\cD}_{P_2}\ldots\otimes\bar{\cD}_{P_n}$ and $\bar{\cD}$ is
isomorphic to the infinite tensor product $ \bigotimes_P \bar{\cD}_P
$.\ecor

\bproof Consider the surjective homomorphism

$$\vp: \bar{\cD}_{P_1}\otimes
\bar{\cD}_{P_2}\ldots\otimes\bar{\cD}_{P_n}\to
\bar{\cD}_{P_1}\bar{\cD}_{P_2}\ldots\bar{\cD}_{P_n}=\bar{\cD}_n$$

which exists by the universal property of the tensor product. The
restriction of $\vp$ to the essential (see the comment after Lemma
\ref{DP}) ideal $\cD_{P_1}\otimes\cD_{P_2}\ldots\otimes\cD_{P_n}$ is
an isomorphism. Thus $\vp$ is an isomorphism. \eproof

From Lemma \ref{specD} it follows that, for each prime ideal $P$, we
have $\Spec \bar{\cD}_P = \Spec \cD_P\cup\Spec \cC = \bigsqcup_n
R/P^n\sqcup R_P$.\mn

For an ideal $I$ in $R$ and $x\in R/I$, we consider the projection
$e^x_I=u^xe_Iu^{-x}$. Since, by Remark \ref{rem1}(b), $e^x_Ie^y_J$
is either zero or equal to $e^z_{I\cap J}$ for $z\in (x+I)\cap
(y+J)$, the set of projections $\{e^x_I\tei\, I\,\mbox{ideal in }
R,\,x\in R/I\}$ is multiplicatively closed.\mn

In particular the set of projections $\{e^x_{P^n}\tei n=0,1,2,\ldots
,\,x\in R/P^n\}$ in $\bar{\cD}_P$ is multiplicatively closed and a
sequence $(\vp_k)$ of characters of $\bar{\cD}_P$ converges to a
character $\vp$ if and only if $\vp_k (e^x_{P^n}) \lori \vp
(e^x_{P^n})$ for each $x$ and $n$. To describe this topology in
terms of a metric we use the norm of an ideal. For an ideal $I$ in
$R$ we denote by $N(I)=|R/I|$ the number of elements in $R/I$. The
topology on $\Spec \bar{\cD}_P$ is described by the metric
$d_\alpha$ defined for any choice of $\alpha >1$ by

$$d_\alpha (\vp ,\psi )= \sum_{n\geq 0,\, x\in
R/P^n}N(P^n)^{-\alpha}|(\vp -\psi )(e^x_{P^n})|
$$

The topology on the first component $\bigsqcup R/P^n$ of $\Spec
\bar{\cD}_P$ is the discrete topology. The elements in this
component are the characters $\eta^x_{P^n}$ uniquely defined by the
condition

$$\eta^x_{P^n}(\delta^x_{P^n})=1 $$

The topology on the second component $R_P$ of $\Spec \bar{\cD}_P$ is
the usual ultrametric topology and finally a sequence
$\eta^{x_k}_{P^{n_k}}$ converges to an element $\eta_z$ in the
second component determined by $z\in R_P$ if and only if
$n_k\to\infty$ and there is $N>0$ such that (using a self-explanatory notation for the image of $z$ in the quotient) $z/P^{n_k}=x_{n_k}$ for
$k\geq N$.\mn

Now, since $\bar{\cD} \cong \bigotimes_P \bar{\cD}_P$, every
character $\vp$ of $\bar{\cD}$ is of the form
$\vp=\bigotimes_P\vp_P$ with each $\vp_P$ either of the form
$\eta_{P^n}^x$ for $n\in\Nz,\,x\in R/P^n$ or $\eta_z$ with $z\in
R_P$.

Again, the set $\{e^x_{I}\tei\, I\mbox{ ideal in }R,\,x\in R/I\}$ of
projections in $\bar{\cD}$ is multiplicatively closed and generates
$\bar{\cD}$. Thus a sequence $(\vp_n)$ of characters converges to
$\vp$ if and only if $\vp_n(e^x_I)\to\vp(e^x_I)$ for each ideal $I$
and $x\in R/I$. This topology is described by the metric $d_\alpha$
defined for any choice of $\alpha >1$ by

$$d_\alpha (\vp ,\psi )= \sum_{I\,\mbox{\footnotesize ideal in}\,R, \, x\in
R/I}N(I)^{-\alpha}|(\vp -\psi )(e^x_{I})| $$

We consider special elements $\eta_I^x$ labeled by $\bigsqcup_IR/I$.
For $I=P_1^{k_1}P_2^{k_2}\ldots P_n^{k_n}$ and $x\in R/I$,
$\eta_I^x$ is defined as

$$\eta_I^x =\bigotimes_{i=1,2,\ldots}\eta_{P_i^{k_i}}^{x_i}$$

Here, $k_i$ is defined to be 0 if $P_i$ does not occur in the prime
ideal decomposition of $I$ and $x_i=x/P_i^{k_i}$.

\bprop\label{49} The subset $\bigsqcup_IR/I$ is dense in $\Spec \bar{\cD}$.
Thus $\Spec \bar{\cD}$ is the completion of  $\bigsqcup_IR/I$ for
the metric $d_\alpha$ described above.\eprop

\bproof It is clear from the discussion above that the set of
elements of the form $\bigotimes_i\eta^{x_i}_{P_i^{k_i}}$ is dense.
We show that each element $\eta =\bigotimes_i\eta^{x_i}_{P_i^{k_i}}$
can be approximated by the $\eta_I^x$. In fact, if
$I=P_1^{k_1}P_2^{k_2}\ldots P_n^{k_n}$ and $x\in R$ is such that
$x/P_i^{k_i}=x_i,\, i=1,\ldots n$, then $\eta (e^y_J)=\eta_I^x
(e^y_J)$ for each ideal $J$ that contains only $P_1,\ldots , P_n$ in
its prime ideal decomposition and for any $y\in R/J$.\eproof

\bremark This description of $\Spec \bar{\cD}$ also clarifies the
wrong-way functoriality in $R$ of the construction. Assume that we
have field extensions $\Qz\subset K\subset L$ and corresponding
inclusions $\Zz\subset R\subset S$ of the rings of algebraic integers.
Denote by $Y_R$ and $Y_S$ the spectra of the corresponding commutative
subalgebras $\bar{\cD}_R$ and $\bar{\cD}_S$ in $\cT[R]$ and $\cT[S]$,
respectively. Thus $Y_R$ and $Y_S$ are completions of the metric spaces
$\bigsqcup_I R/I$ and $\bigsqcup_J S/J$, respectively.

With every character $\eta_J^x\in \bigsqcup_J S/J$ we can associate
a character $(\eta_J^x)'\in \Spec \bar{\cD}_R$ by defining
$(\eta_J^x)' (e_M^y)= \eta_J^x (e_{MS}^y)$ for an ideal $M$ in $R$
and $y\in R/M$. The map $\eta_J^x\to(\eta_J^x)'$ is obviously
contractive (up to a constant $n^\alpha$ with $n=[L:K]$) for the
metrics $d_\alpha$ and thus extends to a continuous map $\Spec
\bar{\cD}_S\to \Spec \bar{\cD}_R$. It is surjective, since the dense
subset $\bigsqcup_I R/I$ of $\Spec \bar{\cD}_R$ has a natural lift
to $\Spec \bar{\cD}_S$. In fact, one immediately checks that
$(\eta_{IS}^x)'=\eta_I^x$ for an ideal $I\subset R$ and $x\in R/I$.
\eremark

\blemma\label{inv} Let $a\in R^\times$ such that $aR=QL$ with $L\in
\cI_n$ and $Q$ relatively prime to each of the $P_1,\cdots,P_n$. Then,
for $I\in \cI_n$,
$$ s_a\delta_{I,n}^0 s_a^* = e_Q\delta_{LI,n}^{0} $$ In particular,
if $aI=bJ$ for two ideals $I$, $J$ in $\cI_n$ and $a,b\in R$, then
$s_a\delta_{I,n}^0s_a^*=s_b\delta^0_{J,n}s_b^*$.\elemma

\bproof Using induction, it suffices to show that, for $aR
=QP_1^{k_1}$ with $Q$ relatively prime to $P_1$,
$$s_a\delta^0_{P_1^{t_1}} s_a^* = e_Q\delta^0_{P_1^{t_1+k_1}}$$
This follows from the following computation

\bglnoz
s_a\delta^0_{P_1^{t_1}} s_a^* = s_a \left( e_{P_1^{t_1}} -
\sum_{x\in
P^{t_1}/P^{t_1+1}}u^xe_{P_1^{t_1+1}}u^{-x}\right)s_a^*\,=\,
e_{aP_1^{t_1}} -
\sum_{x\in P_1^{t_1}/P_1^{t_1+1}}u^{ax}e_{aP_1^{t_1+1}}u^{-ax}\\ =
e_Q\left( e_{P_1^{t_1+k_1}} - \sum_{x\in
P_1^{t_1}/P_1^{t_1+1}}u^{ax}e_{P_1^{t_1+k_1+1}}u^{-ax}\right)
\eglnoz

For the last equality we use the fact that $u^{ax}$ commutes with
$e_Q$ and that $Q$ is relatively prime to $P_1$. \eproof

The dual $\widehat{K^\times}$ of the multiplicative group $K^\times$
acts by automorphisms $\alpha_\chi,\,\chi\in\widehat{K^\times}$ defined
by
$$\alpha_\chi (s_a)=\chi (a)s_a\quad \alpha_\chi (u^x)=u^x\quad
\alpha_\chi (e_I)=e_I$$ By Lemma \ref{elem} (b) the fixed point algebra $B$ is the
subalgebra of $\cT[R]$ generated by all $u^x,\,x\in R$ and $e_I$,
$I$ an ideal in $R$. Integration over $\widehat{K^\times}$ gives a
faithful conditional expectation $\cT[R]\to B$.

On $B$ the dual $\hat{R}$ of the additive group $R$ acts by
automorphisms $\beta_\chi$ given by
$$\beta_\chi (u^x) =\chi(x)u^x \qquad \beta_\chi (e_I)=e_I$$
The fixed point algebra for this action is $\bar{\cD}$. Again,
integration over the compact group $\hat{R}$ defines a faithful
conditional expectation $B\to \bar{\cD}$.

Composing these two expectations we obtain the faithful conditional
expectation $E: \cT[R]\to \bar{\cD}$ which we will use now. Note that, for a typical element $z=s_a^*du^xs_b$, $E(z)=0$, unless $a=b, x=0$ in which case $E(z)=s_a^*ds_a$.

\blemma\label{expect} Let
$$z = d + \sum_{i=1}^m s^*_{a_i}d_iu^{x_i}s_{b_i}$$ be an element of
$\cT[R]$ (cf. Lemma \ref{elem} (b)) such that for each $i$, $a_i\neq
b_i$ or $x_i\neq 0$ and such that $d,d_i\in \bar{\cD}_n$ for some
$n$. Let $n$ be large enough so that also the principal ideals
$a_iR, b_iR$ are in $\cI_n$ for all $i$. Let $\ve> 0$.

\begin{itemize}
  \item[(a)] There is a minimal projection $\delta$ in $\cD_n$ such
  that $\|d\delta\|=\|\delta d\delta\|\geq \|d\|-\ve$.
  \item[(b)] There is $k>0$ and a minimal projection $\delta '\in
  \cD_{n+k}$, $\delta '\leq\delta$ such that $\delta
  's^*_{a_i}d_iu^{x_i}s_{b_i}\delta '=0$ for all $i$.
   \item[(c)] For the projection $\delta '$ in (b) one
   has $\|\delta' z\delta'\|\geq \|E(z)\|-\ve$
\end{itemize}\elemma

\bproof (a) simply expresses the fact that $\cD_n$ is
essential.

(b) Let $\delta =\delta_{I,n}^y$, $I\in \cI_n$, $y\in R/I$. Using Lemma
\ref{inv} we may then choose

$$\delta '=u^y \delta_{I,n}^0\delta^0_{P_{n+1}^{t_1}}\cdots
\delta^0_{P_{n+k}^{t_k}}u^{-y}$$

such that for each $i$ the projections $s_{a_i}\delta's^*_{a_i}$ and
$u^{x_i}s_{b_i}\delta's^*_{b_i}u^{-x_i}$ are orthogonal. This
projection $\delta'$ will have the required properties.

(c) follows immediately from (a) together with (b) using the fact
that $E(z)=d$ and that $d\delta$ is just a multiple of $\delta$
($\delta$ is a minimal projection in $\cD_n$). \eproof

\btheo\label{inj} Let $\alpha :\cT[R]\to A$ be a *-homomorphism into
any C*-algebra $A$. The following are equivalent:
\begin{itemize}
  \item[(a)] $\alpha$ is injective.
  \item[(b)] $\alpha$ is injective on $\bar{\cD}$.
  \item[(c)] $\alpha$ is injective on $\cD_n$ for each $n$.
\end{itemize}\etheo

\bproof (c) implies (b) since  $\cD_n$ is essential in $\bar{\cD}_n$
for each $n$.

(b) $\Rightarrow$ (a): Let $h$ be a positive element in $\cT[R]$
with $\alpha (h)=0$ and $z$ a linear combination as in \ref{expect}
such that $\|h-z\|<\ve$. Let $\delta'$ be a projection as in
\ref{expect} (b) such that $\|\delta' z\delta'\|\geq \|E(z)\|-\ve$
(and such that $\delta'z\delta'$ is a multiple of $\delta'$). If
$\alpha (h)=0$, then $\|\alpha (\delta'z\delta')\|<\ve$ and thus
also $\|\delta'z\delta'\|<\ve$. It follows that $\|E(h)\|<2\ve$.
Since this holds for each $\ve$, $E(h)=0$ and, since $E$ is
faithful, $h=0$.\eproof

From this technical theorem we can derive the following important
corollaries \ref{414},\ref{416} and \ref{417}.

\begin{corollary}\label{414}
The representation $\mu$ of $\cT[R]$ on $\bigoplus _I \ell^2(R/I)$
is an isomorphism.
\end{corollary}
\begin{proof} The restriction of $\mu$ to each $\cD_n$ is injective by Lemma 4.6.
   \end{proof}

Let $\cT\subset \cL(\ell^2(R\rtimes R^\times))$ denote the
C*-algebra defined by the left regular representation of the
semigroup $R\rtimes R^\times$ (cf. section \ref{basic}). Given an
ideal $I$ in $R$ we can define a projection $e'_I$ in
$\cL(\ell^2(R\rtimes R^\times))$ as the orthogonal projection on the
subspace $\ell^2(I\rtimes I^\times)\subset \ell^2(R\rtimes
R^\times)$. Denote by $u'^x$ and $s'_a$ the operators defined by the
left action of $R$ and $R^\times$ on $\ell^2(R\rtimes R^\times)$.
Then it is easy to check that the $u'^x$, $s'_a$ and $e'_I$ satisfy
the relations defining $\cT[R]$.

\blemma\label{frac}
\begin{itemize}
\item [(a)]Every ideal $I$ in $R$ can be written in the form
$\frac{a}{b}R\cap R$ with $a,b\in R^\times$.
\item [(b)]If $I =\frac{a}{b}R\cap R$, then $e_I =s_b^*s_as_a^*s_b$ and,
similarly, $e'_I = s_b^{'*} s_a' s_a^{'*} s_b'$.
\end{itemize}
\elemma

\bproof(a) Let $Q$ and $M$ be ideals such that $I,Q,M$ are
relatively prime and such that $IQ$, $QM$ are principal, say
$IQ=aR$, $QM=bR$. Then $bI=IQM=IQ\cap QM= aR\cap bR$.

(b) One has $bI = aR\cap bR$ and therefore $s_b
e_Is_b^*=s_as_a^*s_bs_b^*$. Since this uses only the relations
defining $\cT[R]$, it also holds in $\cT$.\eproof

This lemma shows that $e'_I\in\cT$, hence we obtain a natural
homomorphism $\cT[R]\to\cT$ by assigning $s_a \mapsto
s'_a,\,u^x\mapsto u'^{x},\, e_I\mapsto e'_I$.

\bcor\label{416} The natural map $\cT[R]\to \cT$ is an isomorphism.\ecor
\begin{proof}
The map is obviously surjective. To prove injectivity, suppose  $I
\in \cI_n$ and $x \in R$ are given. Let $Q\not\in\{P_1,
P_2,\cdots,P_n\}$ be a prime ideal in the ideal class $[I]\inv$ and
let $a $ be a generator  of the principal ideal $IQ$ (for the
existence of such a $Q$ see for instance \cite{Narkiewicz}, chapter
7, \S 2, Corollary 7). Since the exponents of $\{P_1,
P_2,\cdots,P_n\}$ in the prime factorization of $aR$ are identical
to those of $I$, the image of the projection $\delta_{I,n}^x$ fixes
the canonical basis vector $\xi_{(x,a)}\in\ell^2(R\rtimes R^\times
)$, so it does not vanish. Hence the natural map is injective on
$\cD_n$ for each $n$ and therefore injective by Theorem \ref{inj}.
\end{proof}

Denote, as above, by $Y_R$ the
spectrum of $\bar{\cD}$. The semigroup $R\rtimes R^\times$ acts on $Y_R$ in a natural way and
this action corresponds to the canonical action of $R\rtimes R^\times$ on
$\bar{\cD}$ by conjugation by the $u^x$ and the $s_a$. We use the definition of a semigroup crossed product as in \cite{LaRa1}, section 2, \cite{Li}, Appendix A1.

\bcor\label{417} The algebra $\cT[R]$ is canonically isomorphic to
the semigroup crossed product $(\cC(Y_R)\rtimes R)\rtimes
R^\times$.\ecor

\bproof The generators $e_I^0$ of $\cC (Y_R)$ together with the canonical generators of $R\rtimes R^\times$ in  $(\cC(Y_R)\rtimes R)\rtimes
R^\times$ satisfy the relations defining $\cT[R]$, hence they determine a surjective homomorphism $\cT[R]\to (\cC(Y_R)\rtimes R)\rtimes
R^\times$. This homomorphism is injective
on $\bar{\cD}$ and therefore injective by \ref{inj}.\eproof

\section{An alternative description of $\Spec \bar{\cD}$ and the dilation of
$\cT[R]$ to a crossed product by $K\rtimes K^*$}

We will give a parametrization of the spectrum of $\bar\cD$, along
the lines of that obtained for the case $R= \mathbb Z$ in
\cite{LaNe,LaRa}, and use it to realize $\cT[R]$ as a full corner in
a crossed product. Let $\af $ denote the ring of finite adeles over
$K$ and let $\rhat$ be the compact open subring of (finite) integral
adeles; their multiplicative groups are the finite ideles $\afs$ and
the integral ideles $\rhat^*$, respectively.

For each integral adele $a$ and each prime ideal $P$, let
$\epsilon_P(a)$ be the smallest nonnegative integer $n$ such that
the canonical projection of $a$ in $R/P^n$ is nonzero, and put
$\epsilon_P(a) = \infty$ if $a$ projects onto $0 \in R/P^n$ for
every $n$. If $a$ is a finite adele, then there exists $d \in R$
such that $da \in \hat{R}$, and we let $\epsilon_P(a) =
\epsilon_P(da) - \epsilon_P(d)$. This does not depend on $d$. The
group $\rhat^*$ acts by multiplication on $\af$ and the
corresponding orbit space $\af/\rhat^*$ factors as a restricted
infinite product over the set $\cP$ of prime ideals in $R$
 \[
\af /\hat R^* \cong \left\{ \big(\epsilon_P\big)_{P\in \cP} :
\epsilon_P\in \Zz\cup\{\infty\} \text{ and } \epsilon_P \geq 0
\text{ for almost all } P\in \cP\right\}
\]
under the map $a \mapsto  \big(\epsilon_P(a)\big)_{P\in \cP}$. This
product can be viewed as a space of {\em fractional superideals}.
The usual fractional ideals of $K$ appear as the elements $a\rhat^*
\in \af/\rhat^*$ such that $\epsilon_P(a) \in \Zz$ for every $P$ and
$\epsilon_P(a) = 0$ for all but finitely many $P$. The zero divisors
in $\af/\rhat^*$ correspond to sequences for which $\epsilon_P(a) =
\infty$ for some $P$. The elements in $\rhat/\rhat^*$ correspond to
superideals with nonnegative exponent sequences and are analogous to
the usual supernatural numbers (see e.g. wikipedia), in fact
indistinguishable from them as a space -- the distinction will only
arise when we consider the multiplicative action of $K^*$ on
additive classes. For each $a \in \af$ the additive subgroup
$a\rhat$ of $\af$ is invariant under the multiplicative action of
$\rhat^*$.

We will say that two pairs $(r,a)$ and $(s,b)$ in  $\af \times \af$
are equivalent if  $b \in a\rhat^*$ and $s -r \in a\rhat$ and we
will denote by $\omega_{r,a}$ the equivalence class of $(r,a)$.
Since equivalence classes are compact the quotient
\[
\oma  :=  \{ \omega_{r,a} | a\in  \af, r \in \af\}
\]
is a locally compact Hausdorff space, whose elements are pairs
$(r,a)$ with $a\in \af/\rhat^*$ and $r\in \af/a\rhat$.

 When $a$ and $b$ are superideals such that $\epsilon_P(a) \leq \epsilon_P(b)$ for every $P$
 we write $a\leq b$. In this case  $b\rhat \subset a\rhat$ and there is an
  obvious homomorphism  {\em reduction modulo $a$} of
 $\af/b\rhat $ to $ \af/a\rhat$; we will write
 $r(a)$ for the image  of $r\in \af/b\rhat$. When $I$ and $J$ are ideals of $R$ viewed as elements of $\af/\rhat^*$,
  then $I\leq J$ means $J \subset I$
 and the reduction defined above is the usual reduction of ideal classes
 $R/J \to R/I$.\mn

There is a natural action of $\kxkx$ on $\oma$ given by
\bgl\label{action} (m,k) \omega_{r,a} = \omega_{m + kr , ka}. \egl
The additive action $(m,0) \omega_{r,a} = \omega_{m+r,a}$ is by
straightforward addition of classes in $\af/a\rhat$ and  the
multiplicative action $k\,  (a\rhat^*) = (ka)\rhat^*$ on the second
component is also straightforward, but the multiplicative action on
the first component requires
 the homomorphism $\times k : \af/a\rhat \to k\af/ka\rhat= \af / ka\hat{R}$.

Since the set $\rhat$ of integral adeles  is a compact open subset
of $\af$, the subset
\[
\omr := \{ \omega_{r,a}\,|\, r \in \rhat, a\in  \rhat\}
\]
consisting of integral elements is compact open in $\oma$ and is
invariant under the action of $ \rxrx$.

\begin{proposition}
 The projection $\chf_{\omr}$ is full in $\cC_0(\oma) \rtimes \kxkx$ and
 there is a canonical isomorphism
 \[ \cC(\omr) \rtimes \rxrx \cong \chf_{\omr} \left(\cC_0(\oma) \rtimes
 \kxkx \right) \chf_{\omr}.\]
 \end{proposition}
\begin{proof} Clearly  $(\rx)\inv (\rxrx) = \kxkx$, so $\rxrx$ is an
Ore semigroup, and $ \cup_{k\in \rx} (0,k)\inv \omr$ is dense in
$\oma$ because for every element $\omega_{r,a} \in \oma$ there exist
$k\in \rx$ such that $kr \in \rhat$ and $ka \in \rhat/ ka\rhat $. By
\cite[Theorem 2.1]{Laend} the action of $\kxkx$ on $\cC_0(\oma)$ is
the minimal automorphic dilation (see \cite{Laend}) of the semigroup
action of $\rxrx$ on $\cC(\omr)$. The fullness of $\chf_{\omr}$ and
the isomorphism to the corner then follow by \cite[Theorem
2.4]{Laend}.
\end{proof}

\begin{proposition}\label{semixprod}
Let $v_{m,k}$ with $(m,k)\in \rxrx$ be the semigroup of isometries
in $\cC(\omr) \rtimes \rxrx$ implementing the action of $R\rtimes
R^\times$. For each ideal $I$ in $R$ let $E_I$ be the characteristic
function of the set $\{\omega_{s,b} \in \omr\,|\, b \geq I, \, s(b)
\in I\}$. Then the maps $u^x \mapsto v_{x,1}$, $s_k \mapsto
v_{0,k}$, and $e_I \mapsto E_I$ extend to an isomorphism of $\cT[R]$
onto $\cC(\omr) \rtimes \rxrx$.
\end{proposition}
\begin{proof}
The set $\{\omega_{s,b} \in \omr:  b \geq I, \, s(b) \in I\}$ is
closed open because it is defined via finitely many
 conditions  ($\epsilon_P(b) \geq \epsilon_P(I)$ on the prime factors of $I$ and
$s=0 \pmod I$) each of which determines a closed open set; thus
$E_I$ is continuous. The relations Ta are satisfied because $(m,k)
\to v_{m,k}$ is an isometric representation of $\rxrx$, and Tb holds
because $b \geq I $ and $b \geq J$ if and only if $b \geq  I \cap
J$, and $s(I) \in I $ and $ s(J)  \in J$ if and only if $s(I\cap J)
\in I \cap J$. Computing with $m=0$ in equation \eqref{action} shows
that multiplication by $k\in\rx$ maps the support of $E_I$ onto the
support of $E_{kI}$, hence relation Tc holds. Similarly, setting
$k=1$ in equation \eqref{action} shows that addition of $m$ maps the
support of $E_I$ onto itself if $m\in I$, and onto a set disjoint
from it if $m\not\in I$, showing that relation Td holds. This gives
a homomorphism  $h: \cT[R] \to \cC(\omr)\rtimes \rxrx$.

To show that $h$ is surjective it suffices to prove that the
functions $E_I^x := v_{x,1}E_I v_{-x,1}$ separate points in $\omr$.
So let $ \omega_{r,a}$ and $\omega_{s,b}$ be two distinct points in
$\omr$.  If $a \neq b$, we may assume there exists a prime ideal $Q$
such that $\epsilon_Q(a) < \epsilon_Q(b)$ (otherwise reverse the
roles of $a$ and $b$). If we now let $I =  Q^{\epsilon_Q(b)}$, then
$E_I^{s(I)}$ takes on the value $1$ at $\omega_{s,b}$ but vanishes
at $\omega_{r,a}$. If $a = b$ as superideals, since the points
$\omega_{r,a}$ and $\omega_{s,b}$ are distinct, there exists an
ideal $I \leq a$ for which $r(I) \neq s(I)$, in which case the
function $E_I^{s(I)}$ does the separation.

Next we show that this homomorphism is injective on $\cD_n$ for each
$n$. Fix $n$, let $I$ be an ideal whose prime factors are all in
$\{P_1, P_2, \cdots, P_n\}$ and choose a class $x\in R/I$.
 Choose $a\in I$ such that $\varepsilon_{P_j}(a) = \varepsilon_{P_j}(I)$ for $j = 1, 2. \cdots, n$
 (if $I$ is principal, a generator will do; otherwise adjust with a prime ideal
 $Q \notin \{P_1, P_2, \cdots, P_n\}$
 such that $IQ$ is principal). Also choose $r\in \rhat/a\rhat$ such that $r(I) = x$ in $R/I$. Then
  $E_I^x (\omega_{r,a}) = 1$, but $E_{IPj}^x (\omega_{r,a}) = 0$ for each $j$,
  proving that $h(\delta_{I,n}^x ) \neq 0$. Hence $h$ is injective
  on $\cD_n$ and the result follows by Theorem~\ref{inj}.
 \end{proof}

As a byproduct we see that $\omr$ is an `adelic' realization of the
spectrum of $\bar\cD$.
 \begin{corollary}
We view each nonzero ideal $I$ of $R$ as an element $a(I)$ of
$\rhat/\rhat^*$ and, similarly, we view each $x\in R/I$ as a class
$r(x,I)$ in $\rhat/a(I) \rhat \cong R/I$. Then the map $\eta_I^x
\mapsto \omega_{r(x,I),a(I)}$ defined for $(x,I) \in \bigsqcup_I
R/I$ extends to a homeomorphism of the spectrum $Y_R$ of $\bar \cD$
onto $\omr$.
 \end{corollary}

 \begin{proof}
From  the proof of Proposition~\ref{semixprod}, we know that the
 isomorphism $h$ maps the projection $e_I^x$ onto the projection $E_I^x$,
 giving
an isomorphism of $\bar \cD$ to $\cC(\omr)$. To conclude that the
homeomorphism $\hat h\inv:  Y_R \to \omr $ induced by this
isomorphism maps $\eta_I^x$ to  $\omega_{r,a}$ as stated, it
suffices to evaluate
\[\eta_I^x(e_J^y) = \begin{cases} 1& \text{ if } I\subset J \text{ and }
x=y\pmod J\\
0 & \text{ otherwise,}
\end{cases}
\]
and
\[
E_J^y (\omega_{r(x,I),a(I)}) = \begin{cases} 1& \text{ if } J \leq a(I) \text{ and } r(x,I) =y \pmod J\\
0 & \text{ otherwise}
\end{cases}
\]
 for every ideal $J$ in $R$ and $y \in R/J$, and to observe that the two results
coincide because $r(x,I) = x \pmod I$.
\end{proof}
\section{KMS-states for $\beta\leq 2$}\label{beta1}
Recall that for a non-zero ideal $I$ in $R$ we denote by $N(I)$ the
norm of $I$, i.e. the number $N(I)=|R/I|$ of elements in $R/I$. For
$a\in R^\times$ we also write $N(a)=N(aR)$. The norm is
multiplicative, \cite{Neukirch}.

Using the norm one defines a natural one-parameter automorphism
group $(\sigma_t)_{t\in\Rz}$ on $\cT[R]$, given on the generators by
$$\sigma_t(u^x)=u^x\quad\sigma_t(e_I)=e_I\quad\sigma_t(s_a)=N(a)^{it}s_a$$
(this assignment manifestly respects the relations between the
generators and thus induces an automorphism). Recall that a
$\beta$-KMS state with respect to a one parameter automorphism group
$(\sigma_t)_{t\in\Rz}$ is a state $\vp$ which satisfies $\vp
(yx)=\vp (x\sigma_{i\beta}(y))$ for a dense set of analytic vectors
$x,y$ and for the natural extension of $(\sigma_t)$ to complex
parameters on analytic vectors, \cite{BrRo}. Here KMS-states on
$\cT[R]$ are always understood as KMS with respect to the
one-parameter automorphism group $\sigma$ defined above, in which
case the $\beta$-KMS condition for a state $\vp$ on $\cT[R]$
translates to \bgl\label{KMScond} \vp (u^xz)=\vp (zu^x)\quad\vp
(e_Iz)=\vp (ze_I)\quad\vp (s_az)=N(a)^{-\beta}\vp (zs_a)\egl for a
set of analytic vectors $z$ with dense linear span and for the
standard generators $u^x$, $e_I$, $s_a$ of $\cT[R]$. We will usually
choose $z$ to be a product of the form $s^*_bdu^xs_a$ with
$d\in\bar{\cD}$. We will use in the following the notation from
section \ref{D}.

\bprop There are no $\beta$-KMS states on $\cT[R]$ for $\beta<1$.
 \eprop

\bproof Given $a\in \rx$ and $x\in R/aR$, denote by $e_a^x$ the
projection $u^x e_{aR} u^{-x}$. If we apply a $\beta$-KMS state
$\vp$ to the inequality
\[
\sum_{x\in R/aR} e_a^x = \sum_{x\in R/aR} u^x s_a s^*_au^{-x} \leq 1
\]
using that $\vp (e_a^x)=N(a)^{-\beta}$ by \eqref{KMScond} and that
$\abs{R/aR}=N(a)$, we obtain $N(a) N(a)^{-\beta} \leq 1$, which
implies $\beta \geq 1$. \eproof

\blemma\label{dense} Let $\varphi$ be a $\beta$-KMS state for $\beta
> 1$ on $\cT[R]$ and let $\pi_\vp$ be the associated GNS-representation
on $H_\varphi$. Then $\pi_\varphi(\cD_n)H_\varphi$ is dense in
$H_\varphi$, for each $n$.\elemma

\bproof Fix $n\in \Nz$ and let $J$ be an ideal in $\cI_n$. Since the
class group for the field of fractions $K$ is finite, there is $k\in
\Nz$ such that $J^k$ is a principal ideal, say $J^k=aR$ with $a\in
R^\times$. We have $N(J^k) =N(a)$.

The subspace $L=\overline{\pi_\vp(\cD_n)H_\varphi}$ is invariant
under all $u^x$, $x\in R$. It is also invariant under all $s_c,
s_c^*$, $c\in R^\times$. The reason is that  if $cR=QS$ with $S \in
\cI_n$ and $Q$ relatively prime to $P_1, P_2, \cdots, P_n$, then,
according to Lemma \ref{inv}, for every $I\in \cI_n$  we have
$s_c\delta^x_{I}s^*_c = \delta^{cx}_{SI} (u^{cx}e_Q u^{-cx})\leq
\delta^{cx}_{SI}\,\in\cD_n$.

Denote by $E$ the orthogonal projection onto $L^\perp$. Then $1-E$
is the strong limit of $\pi_\vp (h^{1/n})$ where $h$ is a strictly
positive element in $\cD_n$. Therefore $\vp_E$ defined by $\vp_E
(z)=(E\pi_\vp(z)\xi_\vp |\xi_\vp )$, for the cyclic vector $\xi_\vp$
in the GNS-construction, is a $\beta$-KMS functional (consider the
limit $n\to\infty$ of the expression
$\vp((1-h^{1/n})x\sigma_{i\beta}(y))=\vp (y(1-h^{1/n})x)$ using the
fact that $E$ commutes with $y$). Consider the restriction $\rho$ of
$\pi_\vp$ to $L^\perp$. Then $\rho (\cD_n)=0$, whence $\rho
(1-f_{J^k})=0$.

It follows that

$$\rho(1)= \sum_{x\in R/J^k}\rho(u^{x}s_as^*_au^{-x}).$$

This implies that $\vp_E(1)= N(J^k)\vp _E(s_as^*_a)=N(a)\vp
_E(s_as^*_a)$. On the other hand the fact that $\vp_E$ is
$\beta$-KMS implies that $\vp_E(1)= \vp _E(s^*_as_a) =N(a)^\beta\vp
_E (s_as^*_a)$. Since  $\beta >1$, it follows that $\vp_E(1) = 0$
and hence $E=0$. \eproof

\blemma\label{tend} Let $\varphi$ be a $\beta$-KMS state for $1\leq
\beta \leq 2$ on $\cT[R]$ and let $I$ be a fixed ideal in $R$. Then
$\varphi (\delta^0_{I,n})$ tends to 0 for $n \to \infty$.\elemma

\bproof Consider again the GNS-representation $\pi_\vp$ on $H_\vp$.
Let $\tilde{\delta_I}$ denote the limit, in the strong operator
topology, of the decreasing sequence $(\pi_\vp (\delta^0_{I,n}))$ as
$n \to \infty$. By Lemma \ref{prodel}, the projections $\pi_\vp
(u^xs_a)\tilde{\delta_I}\,\pi_\vp (s^*_au^{-x})$ are pairwise
orthogonal for $a\in R^\times /R^*$, $x\in R/aR$.  If we let
$\tilde\vp$ denote the vector state extension of $\vp$ to
$\cL(H_\vp)$ we have

$$\sum_{a\in R^\times/R^*,\,x\in R/aR}\tilde\vp (u^xs_a\tilde{\delta_I}
s^*_au^{-x}) \leq\vp (1)=1.$$

However, since  $\tilde\vp$ is normal we have $\lim_{n\to \infty}
\vp(\delta^0_{IP_1P_2\cdots P_n}) =\tilde\vp(\tilde \delta_I) $
 and since $\vp$ is $\beta$-KMS,
 it follows that $\tilde\vp (u^xs_a\tilde{\delta_I}
s^*_au^{-x})=N(a)^{-\beta}\tilde \vp (\tilde{\delta_I})$. Thus

$$\tilde\vp (\tilde{\delta_I})\sum_{a\in R^\times/R^*}N(a)N(a)^{-\beta}
\leq 1,$$

The series on the left hand side represents the partial Dedekind
$\zeta$-function, corresponding to the trivial ideal class, at
$\beta -1$. Thus, by \cite{Neukirch}, Theorem 5.9, it diverges for
$\beta-1\leq 1$. Therefore the inequality above implies $\tilde\vp
(\tilde{\delta_I})=0$. \eproof

\blemma\label{mult}  Let $\vp$ be a $\beta$-KMS state on $\cT[R]$
and let $I, J\in \cI_n$ be two ideals in $R$ in the same ideal
class. Then, for any $x,y\in R$,
$$\varphi (\delta_{I,n}^x) = N(I)^{-\beta}N(J)^\beta\varphi
(\delta_{J,n}^y)$$ and
$$\varphi (\ve_{IP_1P_2\cdots P_n}) = N(I)^{1-\beta}N(J)^{\beta-1}
\varphi (\ve_{JP_1P_2\cdots P_n}).$$ \elemma

\bproof If $aI=bJ$ then $N(a)N(b)^{-1}=N(I)^{-1}N(J)$ and, by Lemma
\ref{inv}, $s_a\delta_{I,n}^0s_a^*=s_b\delta_{J,n}^0s_b^*$.
Therefore
$$ \varphi (\delta_{I,n}^x)= \varphi (\delta_{I,n}^0)=N(a)^\beta
\varphi (s_a\delta_{I,n}^0s_a^*)= N(a)^\beta\varphi
(s_b\delta^{0}_{J,n}s_b^*)=N(a)^\beta N(b)^{-\beta}\varphi
(\delta^{0}_{J,n})$$

and
$$\varphi (\delta_{I,n}^x) = N(I)^{-\beta}N(J)^\beta\varphi
(\delta_{J,n}^y).$$

Summing over $x\in R/I$ and $y\in R/J$ gives the second statement.
\eproof

\blemma \label{gaugeinv}Let $\varphi$ be a $\beta$-KMS state for
$1\leq \beta \leq 2$ on $\cT[R]$. Let $\bar{\cD}$ be the canonical
subalgebra of $\cT[R]$ generated by all projections $u^xe_Iu^{-x}$
and let $d\in \bar{\cD}$. Then $\varphi (s_a^*du^ys_b)$ is zero
except if $a=b$ and $y=0$ (in which case the argument $s_a^*du^ys_b$
is also an element in $\bar{\cD}$).\elemma

\bproof Suppose first $\beta = 1$ and let $\vp$ be a $1$-KMS state;
then for each $c\in R^\times$ and $x\in R/cR$ we get $\vp (e_c^x)=
\vp(u^x s_c s_c^* u^{-x}) = N(c)^{-1}$, whence $\sum_{x\in R/cR}\vp
(e_c^x)=1$ and

\bgl\label{partition}\vp (z) =\vp \left(\Big(\sum_{x\in
R/cR}e^x_c\Big)z\right) = \vp\left(\sum_{x\in R/cR}e^x_c
ze^x_c\right) \egl

for each $z\in \cT[R]$. Again, let $z=s_{a}^*du^ys_{b}$; then
%(recall that the linear combinations of such elements are dense) we obtain

$$ e_c^x ze_c^x=e_c^xs_a^*du^ys_be_c^x=
s_a^*e_{ac}^{ax}e_{bc}^{bx+y}du^ys_{b}$$

where the product $e_{ac}^{ax} e_{bc}^{bx + y} $ is nonzero if and
only if $(a x + ac R)\, \cap \,(bx +y + b cR) \neq \emptyset$, which
implies
\begin{equation}\label{equationmodc}
(a - b) x \equiv y \mod cR.
\end{equation}
Suppose $z\notin \bar{\cD}$, then either $a \neq b$ or else $a = b$
and $y \neq 0$. In the first case choose $c\in \rx$ with $cR$
relatively prime to $(a - b)R$; then there is a unique $x \in R/cR$
for which (\ref{equationmodc}) holds. In the second case, choose  $c
\in \rx$ with $cR$ relatively prime to $y$; then
(\ref{equationmodc}) has no solutions in $x$. Thus for $z\notin
\bar{\cD}$ there is at most one $x\in R/(cR)$ such that
$\varphi(e_c^x z e_c^x) \neq 0$, and from equation (\ref{partition})
we obtain
\[
|\varphi(z)| \leq N(c)\inv \|z\|.
\]

Since $N(c)$ can be chosen arbitrarily large this shows that $\vp
(z)=0$. This proves that $\vp(z) \neq 0$ only if $a=b$ and $y = 0$
in the case $\beta =1$.

Suppose now that $1<\beta \leq 2$. It follows from  Lemma
\ref{dense} and from the normality of the vector state extension of
$\vp$ to $\cL(H_\vp)$
 that for all $z\in\cT[R]$ and $n\in \Nz$, we have

\bgl\label{expand}
 \vp (z) =\sum_{I\in \cI_n, \,x\in R/I}\vp (\delta_{I,n}^xz\delta_{I,n}^x).
 \egl
Working this out for $z=s_a^*du^ys_b$, where we may assume that $n$
is so large that $aR, bR\in \cI_n$, we find

$$ \delta_{I,n}^xs^*_adu^ys_b\delta_{I,n}^x = s^*_a\delta_{aI,n}^{ax}
du^y\delta_{bI,n}^{bx}s_b
=s^*_a\delta_{aI,n}^{ax}\delta_{bI,n}^{bx+y}du^ys_b$$

and this expression (call it $z_I^x$) does not vanish only if
$aI=bI$, i.e. if $a=gb$ for a unit $g\in R^*$. We have to consider
only that case.

Assume first that $g=1$. Then $z_I^x\neq 0$ only if $bx \equiv
bx+y\mod aI$, that is, only if $y\in bI$.  For a fixed $y\neq 0$
this last condition is satisfied only for the finitely many ideals
$I$ in $R$ such that $bI$ divides $yR$, Thus, if $y\neq 0$, in the
sum (\ref{expand}) there are at most a fixed finite number
(independent of $n$) of non-zero terms   and each individual term is
bounded by $\vp(\delta_{I,n}^0) \|z\|$, which is arbitrarily small
for large $n$ by Lemma \ref{tend}, whence $\vp (z)=0$.

Assume now that  $g\neq 1$. Then $z_I^x\neq 0$ only if $x$ satisfies
$(g-1)bx \equiv y \mod bI$. Let $D:=  \gcd(I, (g-1)R)$. If $y \notin
b D$, then there is no such $x$. Assume thus $y \in b D$, and write
$y = by'$ with $y'\in D$. The nonzero terms in the sum
\eqref{expand} may only arise from $x$ and $I$ such that $(g-1) x
\equiv y' \mod I$. Notice that multiplication by $(g-1)$, viewed as
a map $R/I \to R/I$ is $N(D)$-to-one. Since  $N(g-1) \geq N(D)$,
$N(g-1)$  is a uniform bound on the number of solutions $x$ of the
equation $(g-1) x \equiv y' \mod I$. Thus, for each ideal $I$ there
are at most $N(g-1)$ classes $x+I$ in $R/I$ such that $z_I^{x}\neq
0$. Choosing a reference ideal $J_\gamma$ for each ideal class
$\gamma$ and using Lemma \ref{mult}, we can transform \eqref{expand}
into an estimate

$$\abs{\vp (z)}
\leq\sum_\gamma\sum_{I\in\cI_n\cap\gamma}N(g-1) N(I)^{-\beta}
N(J_\gamma) ^{\beta}\vp (\delta_{J_\gamma,n}^0) \, \|z\| .$$

Since the series for all the partial Dedekind $\zeta$-functions
converge for $\beta >1$ and since each $\vp(\delta_{J_\gamma,n}^0)
\to 0$ as $n\to \infty$ by Lemma \ref{tend}, we conclude that  $\vp
(z) = 0$ unless $a=b$ and $y =0$ also for $1<\beta \leq 2$,
completing the proof.\eproof As a consequence of this lemma, in
order to know $\vp$, it suffices to know its values on $\bar{\cD}$.
Moreover, for $1< \beta \leq 2$  it suffices to know $\vp$ on
$\cD_n$ for all $n$, by Lemma \ref{dense}.

\btheo\label{den} Let $0<\sigma\leq1$. For each ideal class $\gamma$
and for each $n\in \Nz$ we set $\zeta^{(n)}_\gamma
(\sigma)=\sum_{I\in\cI_n\cap \gamma}N(I)^{-\sigma}$. Then for any
two ideal classes $\gamma_1,\gamma_2$ the quotient

$$\frac{\zeta^{(n)}_{\gamma_1} (\sigma)}{\zeta^{(n)}_{\gamma_2} (\sigma)}$$

tends to 1 as $n\to \infty$.\etheo

We postpone the proof and give it in the appendix.

\btheo For each $\beta$ with $1 \leq \beta \leq 2$ there is exactly
one $\beta$-KMS state $\vp_\beta$ on $\cT[R]$; it factors through
the canonical conditional expectation $E:\cT[R] \to \bar{\cD}$ and
it is determined by the values \bgl\label{values} \vp_\beta(e_I^x) =
N(I)^{-\beta} \qquad   \text{ with $I$ an ideal in $R$ and $x\in
R/I$.} \egl For $\beta = 1$ the state $\vp_\beta$ factors through
the natural  quotient map $\cT[R] \to \cA[R]$.\etheo

\bproof Suppose $\vp$ is a $\beta$-KMS state.
 Lemma \ref{gaugeinv} implies that $\vp$
factors through the conditional expectation $E:\cT[R] \to \bar{\cD}$
for $1\leq \beta \leq 2$.

The next step is to show that \eqref{values} holds. Since the linear
combinations of the projections $e_I^x :=u^xe_Iu^{-x}$ are dense in
$\bar{\cD}$ and since $\vp (u^xe_Iu^{-x})=\vp (e_I)$, this will
yield the uniqueness assertion. The argument for $\beta =1$ is
easier and we do it first.

Assume first  $\beta =1$ and let $\vp$ be a $1$-KMS state. If $I$ is
any (non-zero) ideal in $R$, then

$$1=\vp (1)\geq\vp (\sum_{x\in R/I}e_I^x) =N(I)\vp (e_I)$$ and if
$aR\subset I$ and $y\in R/I$, then

$$\vp (e_I^y)=\vp (e_I)\geq \sum_{x\in I/aR}\vp (e_a^x)= (N(a)N(I)^{-1})
N(a)^{-1}=N(I)^{-1}$$ For the identity on the right hand side note
that $\vp (e_a^x)= N(a)^{-1}$ and that $\abs{I/aR}=N(a)N(I)^{-1}$.
We conclude that $\vp (e_I^y)=N(I)^{-1}$, so \eqref{values} holds
for $\beta=1$.

It is obvious that such a state $\vp$ satisfies $\vp(f_P) =1$ and
hence vanishes on the projections of the form $\ve_P$ that generate
the kernel of the quotient map $q : \cT[R]\to \cA[R]$ as an ideal,
so $\vp$ factors through this quotient. It is now easy to prove
existence of a  $1$-KMS state. From Section 4 of  [5], and the fact
that $q$ intertwines the canonical conditional expectations on
$\cT[R]$ and on $\cA[R]$, we know that  the image of $\bar\cD$ in
$\cA[R]$ under  $q$ is naturally isomorphic to $C(\hat R)$
($\hat{R}$ being the profinite completion of $R$). If we let
$\lambda_1$ be the state of $C(\hat R)$ given by normalized Haar
measure on $\hat R$, an easy computation shows that $\vp_1:=
\lambda_1 \circ E \circ q$ satisfies the $1$-KMS condition from
\eqref{KMScond}. This finishes the proof in the case $\beta=1$.

Assume now $1 < \beta \leq 2$ and let $\vp$ be a $\beta$-KMS state.
Using Lemma \ref{prodel}, for the particular element $e_I$ of
$\bar{\cD}_n$, with $I\in\cI_n$, and working in the GNS
representation $\pi_\vp$ of $\vp$ we obtain the formula

\bgl\label{eI} \pi_\vp(e_I ) = \sum_{J\in\cI_n,\, J\subset I,\, x\in
I/J}\pi_\vp( \delta_{J,n}^x ),\egl

with strong operator convergence by Lemma \ref{dense}. We know from
Lemma~\ref{mult} that for $J,L\in \cI_n$ in the same ideal class we
have

$$\vp (\delta_{L,n}^x) = N(L)^{-\beta}N(J)^\beta\vp
(\delta_{J,n}^y)$$

Thus, for an ideal class $\gamma$ in the ideal class group $\Gamma$
for $R$, the expression

$$\alpha^{(n)}_\gamma = N(L)^\beta\vp (\delta^0_{L,n})$$
does not depend on the choice of an ideal $L\in \gamma\cap\cI_n$.
Using the vector state extension $\tilde\vp$ of $\vp$ to handle the
infinite sum we obtain
\bgl\label{eq1}1=\vp(1)=\sum_{I\in\cI_n,\,x\in
R/I}\vp(\delta_{I,n}^x)=\sum_{I\in\cI_n} N(I)\vp(\delta_{I,n}^0)=
\sum_{\gamma\in\Gamma}\alpha^{(n)}_\gamma\zeta^{(n)}_\gamma(\beta
-1)\egl

On the other hand, using (\ref{eI}) and computing with $\tilde\vp$
again, we see that

\bglnoz \vp (e_I) = \sum_{J\subset I,\,x\in I/J}\vp
(\delta^x_{J,n})=\sum_{J\subset I}
\alpha^{(n)}_{[J]}N(I)^{-1}N(J)^{1-\beta}
=\sum_L\alpha^{(n)}_{[LI]}N(I)^{-1}N(LI)^{1-\beta}\\=\sum_L
\alpha^{(n)}_{[LI]} N(I)^{-\beta}N(L)^{1-\beta} =N(I)^{-\beta}\sum_L
\alpha^{(n)}_{[LI]}N(L)^{1-\beta}\\=N(I)^{-\beta}
\sum_{\gamma\in\Gamma} \alpha^{(n)}_{[I]\gamma}\zeta^{(n)}_\gamma
(\beta -1).\eglnoz Dividing by the right hand side of equation
(\ref{eq1}) above and using Theorem \ref{den}, we see that this last
expression converges to $N(I)^{-\beta}$ as $n\to \infty$, proving
that \eqref{values} holds when $1 < \beta \leq 2$.

Let us now prove existence in this case. Since $\cD_n$ is essential
in $\bar{\cD}_n$ there is a natural embedding
$\bar{\cD}_n\hookrightarrow \ell^\infty (\Spec \cD_n)$ and we know
that $\Spec \cD_n=\bigsqcup_{I\in \cI_n}R/I$. The minimal
projections in $\cD_n$ are the $\delta_{I,n}^x,\,I\in \cI_n,\,x\in
R/I$. Thus any $d$ in $\bar{\cD}_n$ is represented by an
$\ell^\infty$-function $(I,x) \mapsto \lambda_I^x( d)$  uniquely
defined by $d \delta_{I,n}^x  = \lambda_I^x( d) \delta_{I,n}^x$ on
$\Spec \cD_n$. Notice that for $d\in \cD_n$ one actually has $d=\sum
\lambda_I^x(d) \delta_{I,n}^x$.

We define a state $\vp_n$ on $\bar{\cD}_n$ by

$$\vp_n (d) =\,\frac{\sum_{I\in \cI_n,\,x\in R/I}\lambda_I^x(d)
N(I)^{-\beta}}{\sum_{I\in \cI_n}N(I)^{1-\beta}}.$$

One obviously has $\vp_n (u^xdu^{-x})=\vp_n (d)$.  Since
$\lambda_I^x(s_a d s_a^*) = \lambda _{J}^{y}(d)$ if $I = aJ$ and $x
= ay$, and  $\lambda_I^x(s_a d s_a^*) = 0$ otherwise, Lemma
\ref{inv} implies that $\vp_n (s_ads^*_a) =N(a)^{-\beta}\vp_n(d)$
for $a$ in $R^\times$ with $aR\in \cI_n$. Since we also have
$\vp_{n+1}|_{\bar{\cD}_n} =\vp_n$, the sequence determines a state
$\vp_\infty$ on
 $\bar{\cD} = \overline{\bigcup_n \bar\cD_n}$. We
define a state $\vp_\beta$ on $\cT[R]$ by $\vp_\beta=\vp_\infty\circ
E$ where $E:\cT[R]\to \bar{\cD}$ is the canonical  conditional
expectation. Then for an element $z=s_a^*du^xs_b$ one has
$$\vp_\beta (s_c z)=N(c)^{-\beta}\vp_\beta (zs_c)\quad \vp_\beta (u^yz)=
\vp_\beta(zu^y )\quad \vp_\beta (d'z)=\vp_\beta (zd')$$ for $c\in
R^\times$, $y\in R$, $d'\in \bar{\cD}$. This suffices to show that
$\vp_\beta$ is $\beta$-KMS. \eproof

\bremark Note that the above construction of the state  $\vp_\infty$
of $\bar\cD$ and thus of $\vp_\beta$ carries through for all
$\beta>1$. Note also that the state $\vp_\infty$ of $\bar\cD$  is
the infinite tensor product state $\vp_\infty =\bigotimes_P\vp_P$
over all prime ideals $P$ in $R$ of the states $\vp_P$ defined on
$\bar{\cD}_P$ by

$$\vp_P (d) =\,\frac{\sum_{n\geq 0,\,x\in R/P^n}\lambda_n^x(d)
N(P)^{-n\beta} }{\sum_{n\geq 0}N(P)^{n(1-\beta)}}$$ for $d=\sum
\lambda^x_n \delta_{P^n}^x$. For $\beta =1$ one takes $\vp_P$ to be
induced from normalized Haar measure on the $P$-adic completion
$R_P$. \eremark

\section{KMS-states for $\beta > 2$}\label{>2}

The basis for the study of KMS-states in this range is the natural
representation $\mu$ of $\cT[R]$ on the Hilbert space

$$ H_R = \bigoplus_{I \textrm{ ideal in } R} \ell^2 (R/I) $$

which has been used already in section \ref{D}.

Let $\cE_I$ denote the projection onto the subspace $\ell^2 (R/I)$
of $H$ and define the operator $\Delta$ on $H$ by $\Delta = \sum_I
N(I)^{-1}\cE_I$. Then $\Delta$ commutes with $\mu(u^x)$ for every
$x\in R$ and $\Delta \mu(s_a) = N(a)\inv \mu(s_a) \Delta$, hence the
dynamics $\sigma $ is implemented spatially by the unitary group
$t\mapsto \Delta^{it}$. Since $\tr (\Delta^\beta ) = \zeta(\beta-1)
$, the operator $\Delta^\beta$ is of trace class  for $\beta >2$ and
$$ \varphi (z) = \Tr (\mu (z)\Delta^\beta)/ \zeta(\beta - 1)$$
defines a $\beta$-KMS state $\varphi$ for each $\beta >2$. Denote by
$\Gamma$ the ideal class group of our number field $K$. The Hilbert
space $H$ splits canonically into a sum of invariant subspaces $H =
\bigoplus_{\gamma\in\Gamma} H_\gamma$, where $ H_\gamma =
\bigoplus_{I \in\gamma} \ell^2 (R/I)$. Denoting by $\mu_\gamma$ and
$\Delta_\gamma$ the restrictions of $\mu$ and $\Delta$ to $H_\gamma$
we obtain a decomposition of $\vp$ as a convex linear combination:
\[
\vp = \sum_{\gamma \in \Gamma} {\textstyle\frac{\zeta_\gamma (\beta
-1) } {\zeta(\beta -1 )}}\vp_\gamma
\]
in which the $\beta$-KMS state $\varphi_\gamma$  associated to the
class $\gamma$ is defined by
$$ \varphi_\gamma (z) = \Tr (\mu (z)_\gamma\Delta_\gamma^\beta)/\Tr
(\Delta_\gamma^\beta),$$

where $\Tr(\Delta_\gamma^\beta)$ is the corresponding  partial zeta
function $\zeta_\gamma(\beta-1)$. We will see below that this family
of $\beta$-KMS states on $\cT[R]$ parametrized by $\Gamma$ consists
of different states.\mn

To obtain the most general KMS-state, we have to consider a more
general family of representations of $\cT[R]$. We fix temporarily a
class $\gamma$ in the class group $\Gamma$ and we choose a reference
ideal $J=J_\gamma$ in this class.

Let $\tau$ be a tracial state on the C*-algebra $C^*(J \rtimes
R^*)$, where the semidirect  product is taken with respect to the
multiplicative action of the group of invertible elements (units)
$R^*$  on the additive group  $J$. Note that these traces form a
Choquet simplex \cite{Thoma}. By \cite[Corollary 5]{nes-tra} the
extreme points can be parametrized by pairs in which the first
component is an ergodic $R^*$-invariant probability measure $\mu$ on
the compact dual group $\hat J$ on which the isotropy is a constant
group $\mu$-a.e., and the second component is a character of that
isotropy group.

Denote by $(H_J,\pi_J,\xi_J)$ the GNS-construction for $\tau$, with
$H_J =L^2(C^*(J\rtimes R^*),\tau)$. If $I$ is another integral ideal
in the class $\gamma$, then there is $a\in K^\times$ such that
$I=aJ$. Multiplication by $a$ induces an isomorphism $J\to I$ which
commutes with the action of $R^*$, hence $(j,g) \mapsto (aj,g)$
induces an isomorphism of groups $J\rtimes R^* \cong I\rtimes R^*$
and of  C*-algebras $C^*(J \rtimes R^*) \cong C^*(I \rtimes R^*)$.
The trace $\tau_a$ on $C^*(I \rtimes R^*)$ obtained from $\tau$ via
this isomorphism is given by $\tau_a(  \delta_{(x,g)}) =
\tau(\delta_{(a\inv x,g)})$ where $\delta_{(x,g)} $ runs through the
canonical generators  of $C^*( I\rtimes R^*)$. If $aJ = bJ$, then
$ab\inv =g \in R^*$ and for every  $\delta_{(j,g')}$ in  $C^*(
I\rtimes R^*)$ we have
\[
\tau_b(\delta_{(j,g')}) = \tau(\delta_{(b\inv j,g')})  = \tau_a(
\delta_{(ab\inv j,g')})  = \tau_a(\delta_{(0,g)} \delta_{(j,g')}
\delta_{(0,g)} ^*)  = \tau_a (\delta_{(j,g')}),
\]
so $\tau_a$ does not depend on the choice of such an $a$, and we
denote it simply as $\tau_I$. From the isomorphism $J\rtimes R^*
\cong I \rtimes R^*$ we obtain an isomorphism $H_J\to H_I$
intertwining the representations $\pi_J$ and $\pi_I$, in which the
cyclic vector $\xi_J\in H_J$ is mapped to the corresponding cyclic
vector $\xi_I\in H_I$.

The representation $\pi_I$ of $C^*(I\rtimes R^*)$ can be induced to
a natural representation (which we also denote $\pi_I$) of $C^*(R
\rtimes R^*)$ on
$$\ell^2 (R/I,H_I)\cong \{f:R\to H_I \big|
f(x+y)=\pi_I(u^x) (f(y)),\,x\in I\}.$$

\begin{lemma}\label{mufromtau}
The direct sum representation $\pi_\tau := \bigoplus_{I \in \gamma}
\pi_I$ of $C^*(R\rtimes R^*)$ on the Hilbert space
\[ H_{\tau} = \bigoplus_{I\in \gamma} \ell^2 (R/I,H_I) \]
extends to a representation of $\cT [R]$ on the same Hilbert space.
\end{lemma}

\begin{proof}
To simplify the notation let $U^x := \pi_\tau (u^x)$ for $x\in R$
and $S_g := \pi_\tau (s_g)$ for $g\in R^*$. We may view the cyclic
vector $\xi_I \in H_I$ as a vector in $\ell^2 (R/I,H_I)$ (supported
on the trivial class) which is  cyclic for the action of
$C^*(R\rtimes R^*)$ on $\ell^2 (R/I,H_I)$.

Next we define $S_a$ for $a\in \rx$.  By \cite[Lemma 1.11]{LavF}
there exists a multiplicative cross section of the quotient $\rx \to
\rx/R^*$ and thus we have a homomorphism $a \mapsto \tilde a$ of
$\rx$ into itself such that for each $a \in \rx$ there exists a
unique $g\in R^*$ with $a = \tilde a g$.

First we define $S_a$ for $a$ in the range of the cross section by
\[
S_{ \tilde a}  ( u^x s_w \xi_I ) := u^{\tilde{a}x} s_w
\xi_{\tilde{a}I}
\]
for $x\in R$ and $w \in R^*$. Since

\bglnoz \big(S_{\tilde a} \sum_i c_i U^{x _i}S_{w_i} \xi_I \,\big|\,
S_{\tilde a} \sum_j c_jU^{x _j}S_{w_j} \xi_I \big) &=& \big(  \sum_i
U^{\tilde a x _i}S_{w_i} \xi_{\tilde aI} \,\big|\, \sum_j U^{\tilde
a x _j}S_{w_j} \xi_{\tilde aI}  \big)\\ &=&  \sum_{i,j} \big(
S_{w_j}^*U^{-\tilde a x _j}U^{\tilde a x _i}S_{w_i} \xi_{\tilde aI}
\,\big|\,   \xi_{\tilde aI}  \big)
\\ &=&\sum_{\{i,j: x_i - x_j \in I\}} \tau_{\tilde aI} ( u^{\tilde  a
(x _i - x_j)}s_{w_i}s_{w_j}^*)
\\ &=& \sum_{\{i,j: x_i - x_j \in I\}} \tau_I( u^{(x _i - x_j)}s_{w_i}
s_{w_j}^*)
\\ &=& \big(  \sum_i c_i U^{x _i}S_{w_i} \xi_I \,\big|\,  \sum_j c_j
U^{x _j}S_{w_j} \xi_I \big) \eglnoz

because $\tau_I$ and $\tau_{\tilde a I}$ are traces satisfying
$\tau_{\tilde aI} (u^x s_w) = \tau_I (u^{\tilde a \inv x} s_w)$, the
map $S_{\tilde a}$ is isometric on a dense set and thus extends
uniquely  by linearity and continuity to an isometry $S_{\tilde a}$
of $\ell^2 (R/I,H_I)$ into $ \ell^2 (R/\tilde{a}I,H_{\tilde{a}I})$.
 For  general $a\in \rx$ we simply write $a =  \tilde a g$ and we let
$ S_a  :=  S_{\tilde a}S_g $.

For an ideal $L\subset R$, we view $\ell^2(L/I,\;H_I)$ as the
obvious subspace of $\ell^2(R/I, H_I)$ and we define $E_L$ to be the
orthogonal projection onto $\bigoplus_{I\in\gamma,I\subset L}
\ell^2(L/I,\;H_I)$.

It is easy to verify that $S$ is a representation of the semigroup
$\rx$ by  isometries and  that $E$ is a family of projections
representing the lattice of ideals of $R$, such that  $U$, $S$, and
$E$ satisfy the relations defining $\cT[R]$. Hence there is a
representation $\mu_\tau$ of $\cT[R]$ such that $\mu_\tau(u^x) =
U^x$, $\mu_\tau(s_a) = S_a$ and $\mu_\tau(e_I ) = E_I$.
\end{proof}

We  can use the representation $\mu_\tau$ to define a $\beta$-KMS
state as follows. Let $\cE_I$ denote the orthogonal projection onto
the subspace $\ell^2(R/I,H_I)$ and define a positive operator
$\Delta$ on $H_\tau$ by $\Delta=\sum N(I)^{-1}\cE_I$. Since $\Delta$
commutes with $U^x$ and with $E_I$, and since $\Delta S_a = N(a)S_a
\Delta$, the unitary group $t\mapsto \Delta^{it}$ implements the
dynamics, just as in our initial example, but when $H_{J_\gamma}$ is
not finite dimensional, the operator $\Delta^\beta$ is not of trace
class. Nevertheless,  we have $\sum_{I\in\gamma,\,x\in R/I}
(\Delta^\beta U^x\xi_I\,|\,U^x\xi_I) = \zeta_\gamma(\beta-1)$, and
setting

\bgl\label{kmsfromtau} \vp_{\gamma,\tau} (z)=
\frac{\sum_{I\in\gamma,\,x\in R/I} (\mu_\tau (z)\Delta^\beta
U^x\xi_I\,|\,U^x\xi_I)}{\sum_{I\in\gamma,\,x\in R/I} (\Delta^\beta
U^x\xi_I\,|\,U^x\xi_I)} \egl

yields a  $\beta$-KMS state for each $\beta >2$, by
\eqref{KMScond}.\mn

As before,  let $P_1, P_2, \ldots $ be an enumeration of the prime
ideals in $R$. When $\rho$ is a given representation of $\cT[R]$,
for each ideal $I$ in $R$, let $\tilde{\ve}_I$ denote the strong
operator limit of the decreasing sequence of projections
$\rho(\ve_{I P_1 P_2 \cdots P_n})$. The $\tilde{\ve}_I$ form a
family of pairwise orthogonal projections. Similarly, let $\tdi$ be
the strong limit of the decreasing family of projections
$\rho(\delta_{I,n}^0)$, as in the proof of Lemma~\ref{tend}, and let
$\tilde\delta_I^x:= \rho(u^x)\tilde\delta_I \rho(u^{-x})$. If $I$
and $L$ are ideals in $R$, then

\bgl\label{rem}
 \tilde{\delta}^x_I\rho(e_L)=\begin{cases}
 \tilde{\delta}^x_I & \text {if $I\subset L$ and $x\in L/I$}\\
 0 & \text{ otherwise,}
 \end{cases}
\egl

To see why, observe that as soon as $n$ is large enough that $\{P_1,
P_2,\cdots, P_n\}$ contains all the prime factors of $I$ and $L$,
$\delta_{I,n}^0 e_L=  \delta_{I,n}^0$ if $I\subset L$ and $x\in
L/I$, and is $0$ otherwise, from the description of $\bar{\cD}_n$ in
Lemma~\ref{prodel}.

\begin{lemma}\label{taufromphi}
Let $\mu_\tau$ be the representation constructed in Lemma
\ref{mufromtau} from a trace $\tau$ on $C^*(J_\gamma\rtimes R^*)$,
and let $U^x := \mu_\tau(u^x)$ and $S_a := \mu_\tau(s_a)$. Suppose
$I$ is an  ideal in $R$ and $x\in R$;
\begin{enumerate}
\item[(i)] if $I \in \gamma$, then $U^x\tdi U^{-x} = U^x E_I \cE_I U^{-x}$, the projection onto $U^xH_I$;
\item[(ii)] if $I \not\in \gamma$, then $U^x\tdi U^{-x} = 0$; and
\item[(iii)]  the trace $\tau$ is retrieved from $\vp_{\gamma, \tau}$ by conditioning to $\tdj$:
\[
 \tau (u^x s_g)  =N(J_\gamma)^\beta\zeta_\gamma(\beta-1) \vp_{\gamma,\tau} (\tdj U^x S_g \tdj) \qquad x\in J_\gamma,  \ \ g\in R^*.
\]
\end{enumerate}
\end{lemma}
\begin{proof}
For part (i),  notice  that when $I\in \gamma$, then $\cE_I = \tilde
\varepsilon_I := \lim_{n\to \infty} \mu_\tau
(\varepsilon_{IP_1P_2\cdots P_n})$, then multiply by $E_I$ and
translate with $x\in R/I$.

For part (ii), recall that if $I$ and $I'$ are different ideals,
then the projections $\tdi^x$ and $\tilde\delta_{I'}^{x'}$ are
mutually orthogonal. Since the Hilbert space $H_\tau  = \bigoplus_{I
\in \gamma} \ell^2 (R/I,H_I) $ is generated by the ranges of the
projections $\tdi^x$ with $I\in \gamma$ and $x\in R/I$, it follows
that  $\tilde\delta_{I'}^{x'} =0$ whenever $I' \not \in \gamma$.

Finally, notice that $H_{J_\gamma}$ viewed as a subspace of $H_\tau$
is invariant for the action of $C^*(J_\gamma\rtimes R^*)$ and, by
construction, the restriction of $\mu_\tau$  to $C^*(J_\gamma\rtimes
R^*)$  and to this subspace is the GNS representation of $\tau$,
with cyclic vector $\xi_{J_\gamma}$. Since $\tdj =E_{J_\gamma}
\cE_{J_\gamma}$ is the projection onto $H_{J_\gamma}$, the sum in
equation \eqref{kmsfromtau} has only one term, giving the identity
in part (iii).
\end{proof}

It turns out that to parametrize the $\beta$-KMS states in the
region $\beta >2$ all we need to do is combine states constructed
from different ideal classes. \btheo Suppose $\beta >2$ and  choose
a fixed reference ideal $J_\gamma\in\gamma$ for each $\gamma$ in the
class group $\Gamma$ of $K$. For each tracial state $\tau$  of
$\bigoplus_\gamma C^*(J_\gamma\rtimes R^*)$ write $\tau = c_\gamma
\tau_\gamma$ as a convex linear combination of traces on the
components and define $\vp_\tau := \sum_\gamma c_\gamma
\vp_{\gamma,\tau_\gamma}$ using equation \eqref{kmsfromtau}. Then
the  map $\tau \mapsto \vp_\tau$  is a continuous  affine
isomorphism of the Choquet simplex of tracial states of
$\bigoplus_\gamma C^*(J_\gamma\rtimes R^*)$ onto the simplex of
$\beta$-KMS states for $\cT[R]$.

Going in the opposite direction, the $\gamma$-component of the trace
$\tau$ corresponding to a given $\beta$-KMS state $\vp$ is obtained
by conditioning (the vector state extension of) $\vp$ to  $\tdj$,
\[c_\gamma\tau_{\vp,\gamma}(u^xs_g)  := N(J_\gamma)^\beta \zeta_\gamma
(\beta-1)\tilde\vp(\tdj U^x S_g \tdj),\] where $c_\gamma :=
N(J_\gamma)^\beta \zeta_\gamma(\beta-1)\tilde\vp(\tdj)$. \etheo
\begin{proof}
Since  $\vp_{\gamma,\tau_\gamma}$  is a $\beta$-KMS state for each
$\gamma$, so is $\vp_\tau = \sum_\gamma c_\gamma
\vp_{\gamma,\tau_\gamma}$. To see that $\tau$ is obtained from
$\vp_\tau$ by conditioning to  $\tdj$, assume $c_\gamma \neq 0$
(otherwise skip $\gamma$). Then Lemma~\ref{taufromphi}(iii) implies
that
\[
c_\gamma\tau_\gamma(u^x s_g) = N(J_\gamma)^\beta \zeta_\gamma(\beta
-1) \vp_{\gamma,\tau_\gamma}(\tdj U^x S_g \tdj)
\]
which is equal to $ N(J_\gamma)^\beta \zeta_\gamma(\beta -1)
\vp_\tau(\tdj U^x S_g \tdj)$ by Lemma~\ref{taufromphi}(ii). This
proves  that the map $\tau \mapsto \vp_\tau$ is injective. Next we
show it is surjective.\mn

Suppose $\vp$ is a $\beta$-KMS state and let $\cT[R]$ be represented
on $H_\vp$ in the GNS-construction for $\varphi$. As usual, we
denote by $\tilde\vp$ the vector state extension of $\vp$ to
$\cL(H_\vp)$, and we also write $\pi_\vp(u^x) = U^x$,  $\pi_\vp(s_a)
= S_a$ for simplicity of notation.

We show next that $\bigoplus_{I} \tilde{\ve}_I (H_\vp)=H_\vp$. Using
Lemma \ref{dense} we obtain, as soon as $I$ is in the semigroup
$\cI_n$ generated by $P_1,P_2,\cdots, P_n$,

$$ \vp (1) = \sum_{I\in \cI_n}\vp (\ve_{I P_1
P_2 \cdots P_n})=\sum_{\gamma\in \Gamma}\sum_{ I\in \cI_n\cap
\gamma}\vp (\ve_{I P_1P_2\cdots P_n})$$

where, according to Lemma \ref{mult} $$\vp (\ve_{I P_1P_2\cdots
P_n})= N(J_\gamma)^{\beta-1} N(I)^{1-\beta}\vp (\ve_{J_\gamma
P_1P_2\cdots P_n})$$

for $I\in \cI_n\cap \gamma$. In the limit $n\to \infty$ this gives
\[
\vp (1) = \sum_{\gamma\in\Gamma}N(J_\gamma)^{\beta -1}\zeta_\gamma
(\beta-1) \tilde\vp (\tilde{\ve} _{J_\gamma}) =
\sum_{\gamma\in\Gamma}N(J_\gamma)^{\beta }\zeta_\gamma (\beta-1)
\tilde\vp (\tdj)
\]
where $\zeta_\gamma $ is the partial $\zeta$-function $\zeta_\gamma
(t) = \sum_{I\in \gamma}N(I)^{-t}$, which converges for $t> 1$.\mn

Let $F$ denote the orthogonal complement of $\bigoplus_{I}
\tilde{\ve}_I (H_\vp)$ and $\psi$ the restriction of $\tilde\vp$ to
$\pi_\vp(\cT[R])|_F$. Since $\psi$ is again a $\beta$-KMS functional
(see \cite{BrRo}, 5.3.4 and 5.3.29) and since the $F  \tilde\ve
_{J_\gamma} F=0$, the above identity applied to $\psi$ shows that
$\psi (F)=0$ and thus $F=0$, proving $\bigoplus_{I} \tilde{\ve}_I
(H_\vp)=H_\vp$. \mn

Since $\sum_{\gamma\in\Gamma}\sum_{I\in \gamma} \sum_{x\in R/I}
\tilde\delta_I^x = \sum_{\gamma\in\Gamma}\sum_{I\in \gamma}
\tilde\ve_I = 1$ in the GNS representation of $\vp$ and since
$\tilde\delta_I^x$ is in the centralizer of $\tilde\vp$, we have
\bgl\label{reconstruction} \tilde\vp (\,\cdot\,)  = \tilde\vp
(\;\cdot\, \sum_{\gamma, I, x} \tilde\delta_I^x) =  \sum_{\gamma \in
\Gamma}\sum_{I\in \gamma} \sum_{x\in R/I} \tilde\vp(
\tilde\delta_I^x \,\cdot\, \tilde\delta_I^x). \egl

The projection  $\tdi$ commutes with $U^x S_g$ for $x\in I$ and
$g\in R^*$, hence the canonical map $u^xs_g \mapsto \tdi U^x
S_g\tdi$  determines a homomorphism of $C^*(I\rtimes R^*)$ to the
corner $\tdi \cT[R] \tdi$. This homomorphism is surjective because
\[
\tdi S_a^* d U^y S_b \tdi =S_a^* \tilde\delta_{aI}
\tilde\delta_{bI}^y  d  U^y  S_b
\]
is nonzero only if $aI = bI$ and $y \in aI = bI$, i.e. only if $b =
ga$ for some $g\in R^*$ and $y' = y/a \in I$, in which case the
whole expression reduces to $\tdi    z_d  U^{y'}  S_a^* S_b = z_d
\tdi   U^{y'} S_g$, where $z_d$ is a scalar.

Next we show how to recover each $\tilde\vp( \tdi^x
\,\cdot\,\tdi^x)$ from  $\tilde\vp( \tdj^x \,\cdot\, \tdj^x)$. For
each  $I \in \gamma$  there exist $a_I$ and $b_I$ in $\rx$ such that
$(a_I/b_I) J_\gamma = I$. By Lemma~\ref{inv}
\[
 \tilde{\delta}_I^x := U^x     \tilde\delta_{I}   U^{-x}  = U^x
 S^*_{b_I}S_{a_I} \tdj S_{a_I}^* S_{b_I}  U^{-x} .
\]

Using  the KMS-condition we obtain
 \bglnoz
 \tilde\vp( \tilde\delta_I^x \,\cdot\, \tilde\delta_I^x) &=&
\tilde\vp( U^x S^*_{b_I}S_{a_I} \tdj S_{a_I}^* S_{b_I}  U^{-x}
\,\cdot \, U^x    S^*_{b_I}S_{a_I} \tdj S_{a_I}^* S_{b_I}  U^{-x}  )\\
&=& N(a_I/b_I)^{-\beta}\tilde\vp( \tdj S_{a_I}^* S_{b_I}  U^{-x}
\,\cdot\,
U^x    S^*_{b_I}S_{a_I} \tdj )\\
&=& N(I)^{-\beta} N(J_\gamma)^\beta \tilde\vp(  \tdj S_{a_I}^*
S_{b_I}  U^{-x} \,\cdot\,  U^x    S^*_{b_I}S_{a_I} \tdj  ). \eglnoz
Notice that the choice of $a_I$ and $b_I$ does not affect the result
because of Lemma \ref{unit} and because for $g\in R^*$ the unitary
$S_g$ commutes with $\tdj$ and centralizes $\vp$. Hence every
$\beta$-KMS state for $\beta>2$ is determined by the collection of
conditional functionals $\{ \tilde \vp(\tdj \, \cdot \, \tdj)
\,\big|\,\gamma \in \Gamma\}$.

 The state $\vp$ gives rise to traces as follows. First let $c_\gamma := N(J_\gamma)^\beta \zeta_\gamma(\beta-1)\tilde\vp(\tdj)$
and recall that $\sum_\gamma c_\gamma =1$ from above. When $c_\gamma
\neq 0$,  set
  \[
 c_\gamma \tau_{\gamma,\vp} (u^x s_g) := N(J_\gamma)^\beta \zeta_\gamma(\beta-1)\tilde\vp(\tdj U^xS_g \tdj),
 \]
 which defines a tracial state $\tau_{\gamma,\vp}$ on $C^*(J_\gamma \rtimes R^*)$, by the KMS condition.
 This  shows that the given $\beta$-KMS state $\vp$ arises as $\vp_\tau$ from the trace $\tau := \sum_\gamma c_\gamma \tau_{\gamma,\vp}$ that it determines on $\bigoplus_{\gamma\in \Gamma}C^*(J_\gamma\rtimes R^*)$, proving the surjectivity of the
 map $\tau \mapsto \vp_\tau$.

The map $\vp \mapsto \tau$ is clearly affine and continuous in the
weak*-topology, and since the spaces of traces and of $\beta$-KMS
states are compact Hausdorff, the map is a homeomorphism.
\end{proof}

\bremark
\begin{enumerate}
\item Our parameter space of traces is obviously not canonical because it depends on the arbitrary choice of representative ideals $J_\gamma$ in each class. However, the traces are determined up to canonical isomorphisms of the underlying C*-algebras, as discussed at the beginning of the section.

\item  The $\beta$-KMS states can be evaluated explicitly on products of the form $s_a^* e^z_J u^y s_b$; since these have dense linear span, this characterizes $\vp_\tau$.
Assume first $\tau$ is supported on a single ideal class $\gamma \in
\Gamma$. By \eqref{kmsfromtau} we may assume $a\inv b = g\in R^*$,
for otherwise $\vp_\tau(s_a^* e^z_J u^y s_b) =0$. Then \bglnoz
\qquad\vp_{\gamma,\tau}  (s_a^* e^z_J u^y s_b) &=&
\frac{1}{\zeta_\gamma(\beta-1)}\sum_{I\in\gamma}\sum_{x\in R/I}
(S_a^* E^z_J U^y S_b
\Delta^\beta  U^x\xi_I\,|\,U^x\xi_I)\\
&=& \frac{1}{\zeta_\gamma(\beta-1)}\sum_{I\in\gamma}\sum_{x\in R/I}
(U^{-x}S_a^* E^z_J U^y S_bU^x
\Delta^\beta \xi_I\,|\,\xi_I)\\
&=&
\frac{1}{\zeta_\gamma(\beta-1)}\sum_{I\in\gamma}\sum_{x\in R/I} N(I)^{-\beta}(U^{-x}S_a^* E^z_J U^y S_bU^x \xi_I\,|\,\xi_I)\\
&=&
\frac{1}{\zeta_\gamma(\beta-1)}\sum_{I\in\gamma}\sum_{x\in R/I} N(I)^{-\beta}(S_a^* U^{-ax}E^z_J U^{y + agx}S_g S_a\xi_I\,|\,\xi_I)\\
&=&
\frac{1}{\zeta_\gamma(\beta-1)}\sum_{I\in\gamma}\sum_{x\in R/I} N(I)^{-\beta}( E^{z-ax}_J U^{y+ agx- ax}S_g \xi_{aI}\,|\,\xi_{aI}).\\
 \eglnoz

\noindent The nontrivial contributions come from terms with
\begin{itemize}
\item[]$z-ax \in J$,
\item[]  $y+ a(g-1)x\in {aI}$ and
\item[]  $aI \subset J$.
\end{itemize}
 Thus, recalling that $\xi_{aI}$ is the cyclic vector for the
GNS representation of $\tau_I$ (the notation is from the
construction leading up to Lemma~\ref{mufromtau}),
 the sum reduces to
\[
\vp_{\gamma,\tau}  (s_a^* e^z_J u^y s_b) =
\frac{1}{\zeta_\gamma(\beta-1)}\sum_{I\in\gamma,\, aI \subset
J}\sum_{x\in P_I} N(I)^{-\beta}\tau_{aI}( u^{y+ ax(g-1) }s_g)
\]
where $P_I := \{x\in R/I\, \big| \, ax-z\in J/I, \, y+ax(g-1) \in
I\}$.

\noindent If we now start with a trace $\tau = \sum_{\gamma \in
\Gamma} c_\gamma \tau_\gamma$, then the values of the corresponding
$\beta$-KMS state are given by
\[
\vp (s_a^* e^z_J u^y s_b) = \sum_{\gamma\in
\Gamma}\sum_{I\in\gamma,\, aI \subset J}\sum_{x\in P_I}
\frac{c_\gamma N(I)^{-\beta}}{\zeta_\gamma(\beta-1)} \tau_{\gamma,
\,aI}( u^{y+ a(g-1)x }s_g) .
\]

\item
The $\infty$-KMS states are, by definition, the weak-* limits as
$\beta \to \infty$ of $\beta$-KMS states, and they too can be
computed explicitly, by taking limits in the above formula. Notice
that $\frac{N(I)^{-\beta}}{\zeta_\gamma(\beta -1)} \to 0$ as $\beta
\to\infty$, except when $I$ is norm-minimizing in its class, in
which case the limit is $k_\gamma\inv$ (with $k_\gamma$ the number
of norm-minimizing ideals in the class $\gamma$). Thus, $\infty$-KMS
states are still indexed by traces $\tau=\sum_\gamma c_\gamma
\tau_\gamma$  of $\bigoplus_\gamma C^*(J_\gamma\rtimes R^*)$, and
are given by
\[
\vp (s_a^* e^z_J u^y s_b) = \sum_{\gamma\in
\Gamma}\sum_{I\in\underline\gamma, \, aI \subset J}\sum_{x\in P_I}
{c_\gamma k_\gamma\inv} \tau_{\gamma, \,aI}( u^{y+ a(g-1)x }s_g) .
\]
where the sum is now over the subset $\underline{\gamma} $ of
norm-minimizing ideals in $\gamma$.
\end{enumerate}\eremark\

\bremark\label{Rtimes}
 As a much simpler ``toy model" for the dynamical system
 $(\cT[R], (\sigma_t))$ we can also consider the Toeplitz algebra
 $\cT[\rx]$ associated with the multiplicative
semigroup $\rx$ of $R$, i.e. the C*-algebra generated by the left
regular representation of this semigroup. It is generated by
isometries $s_a$, $a \in \rx$ and carries an analogous one-parameter
automorphism group $(\sigma^\times_t)$ defined by
$\sigma^\times_t(s_a) = N(a)^{it}s_a$. Since $R^\times$ is a split
extension of $R^\times/R^*$ by $R^*$, \cite[Lemma 1.11]{LavF}, we
see that $\cT[R^\times]$ is the tensor product of $C^*(R^*)$ and the
Toeplitz algebra $\cT[R^\times /R^*]$ for the semigroup $R^\times
/R^*$ of principal integral ideals. In the case where $R$ is a
principal ideal domain, $\cT[\rx]$ is then simply an infinite tensor
product of the ordinary Toeplitz algebras (i.e. universal
C*-algebras generated by a single isometry) generated by the
isometries associated to the primes in $R$, and of $C^*(R^*)$. In
this case the situation is nearly trivial. An easy exercise shows
that the KMS-states for each $\beta>0$ are labeled by the states of
$C^*(R^*)$.

However, in the case of a non-trivial class group, we obtain a
non-trivial C*-dynamical system, essentially, because there is an
`interaction' between the classes. The methods and results of the
last two sections (including Theorem 5.6) immediately lead to a
determination of its KMS structure. One finds that for $\beta = 0$
there is a family of $0$-KMS states ($\sigma$-invariant traces)
indexed by the $\sigma$-invariant states on $C^*(K^\times)$ (such a
state has to factor through the quotient of $\cT[R]$ where each of
the generators $s_a$ becomes unitary - this quotient is exactly
$C^*(K^\times)$). For each $\beta$ in the range $0 < \beta \leq 1$
the $\beta$-KMS states correspond exactly to the states of
$C^*(R^*)$ (there is a unique $\beta$-KMS state on $\cT[R^\times
/R^*]$ which can be combined with an arbitrary state on the tensor
factor $C^*(R^*)$). For each $\beta$ in the range $1 < \beta <
\infty$ the simplex of KMS states splits in addition over the class
group $\Gamma$. Thus the KMS states in that range are labeled by the
states of $C^*(R^*\times\hat{\Gamma})$.

We note that it is known that the class group $\Gamma$ for $K$ is
determined already by the semigroup $R^\times$. In fact $\Gamma$
coincides with the semigroup class group defined by the ideals in
this semigroup (i.e. the subsets invariant under multiplication by
all elements), cf. \cite[section 2.10]{GeHa}.\eremark

\section{Ground states}\label{gs}

Recall that  a state $\varphi$ on a C*-dynamical system $(\mathfrak
B, (\sigma_t)_{t\in \Rz})$ is a ground state if and only if the
function
\[
z \mapsto \varphi( w\, \sigma_z( w') )
\]
is bounded on the upper half plane on a set of analytic vectors
$w,w'\in\mathfrak B$ with dense linear span.

\begin{proposition}\label{tent}
Let  $\varphi$ be a state of $\cT[R]$. Then the following are
equivalent:
\begin{enumerate}
\item $\varphi$ is a ground state;
\item  for all $d\in\bar{\cD}$,
$a,b\in R^\times$, $x\in R$ and $w\in \cT[R]$ we have $\varphi( w\,
s_a^* d u^x s_b) = 0 $, whenever $N(a)> N(b)$;
\item for $a,b\in R^\times$, $x\in R$, we have $\vp (s_b^*u^xs_as_a^*
u^{-x}s_b)=0$, whenever $N(a)> N(b)$ (note that the expression under
$\vp$ depends on $x$ only via its image in $R/aR$);\end{enumerate}
\end{proposition}
\begin{proof}
$\varphi$ is a ground state if and only if the function
\[
z \mapsto \varphi( w\, \sigma_z( w') )
\]
is bounded on the upper half plane on a set of analytic vectors
$w,w'\in\cT[R]$ with dense linear span. We may choose $w'$ of the
form $s_a^*  du^x s_b$.

We have
\[
\varphi( w\, \sigma_z( s_a^*  d u^x s_b) ) = N(b/a)^{iz} \varphi( w
\,( s_a^*  d u^x s_b) ),
\] This function is bounded on the upper half plane if and only if it
vanishes when $N(b/a) <1$. This proves that (1) and (2) are
equivalent.

By the Cauchy-Schwarz inequality (2) is equivalent to the fact that
$$\varphi( ( s_a^*  d u^x s_b)^*( s_a^*  d u^x s_b) )=0$$ for all
$a,b,x,d$. However $( s_a^*  d u^x s_b)^*( s_a^*  d u^x s_b)\leq
\|d\|^2s_b^*u^xs_as_a^* u^{-x}s_b$. This shows that (2) and (3) are
equivalent.\eproof

We will see that the ground states on $\cT[R]$ are supported on
projections corresponding to what we call ``norm-minimizing
ideals''. We say that an ideal $I$ in $R$ is norm-minimizing if for
any other ideal $J$ in the same ideal class we have $N(I)\leq N(J)$.
The use of norm-minimizing ideals was suggested by work in
preparation by Laca-van Frankenhuijsen.

Recall from Lemma \ref{frac} (a) that every ideal $I$ in $R$ can be
written in the form $\frac{a}{b}R\cap R$ with $a,b\in R^\times$.
\begin{lemma}\label{mini}\begin{enumerate}
\item[(i)]If a product $J = I L$ is norm-minimizing, then so are $I$
and $L$.
\item[(ii)] The prime ideals that are norm-minimizing generate the ideal
class group.
\item[(iii)] If $I = \frac{a}{b}R\cap R$ is norm-minimizing, then
$N(a)\leq N(b)$.
\end{enumerate}
\end{lemma}
\begin{proof}
The proof of part (i) is obvious and (ii) follows easily from (i).
To prove (iii) observe that for each $I=\frac{a}{b}R\cap R$, the
integral ideal $I'=\frac{b}{a}I=R\cap\frac{b}{a}R$ is in the same
class and $N(I')=N(b)N(a)^{-1}N(I)$. Thus, if $I$ is norm
minimizing, necessarily $N(b)N(a)^{-1}\geq 1$
\end{proof}

\begin{lemma}\label{ground-notnm} Let $\vp$ be a ground state of $\cT[R]$. Then
$\vp (e_I^x)=0$ for each ideal $I$ in $R$ which is not
norm-minimizing and for each $x\in R/I$.
\end{lemma}

\bproof If $I$ and $J$ are two ideals in the same ideal class, then
there exist integers $a$ and $b$ in $R^\times$ such that $bI = aJ$,
so $e_I = s_b^* s_a e_J s_a^* s_b$. Assuming that $J$ is
norm-minimizing but $I$ is not, then $N(a)> N(b)$ by
Lemma~\ref{mini}(iii), so we may use part (2) of Proposition
\ref{tent} on the product $(u^x s_b^* s_a e_J) \, (s_a^* s_b
u^{-x})=e_I^x$ to finish the proof. \eproof In particular, the above
proposition implies that $\vp (e_P^x)=0$ for each prime ideal $P$
which is not norm-minimizing and for each $x\in R/P$. Thus $\vp
(\ve_P)=\vp (1- f_P)=1$ for such ideals. To take advantage of this
feature, we will order the prime ideals in $R$ in such a way that
$P_1,\ldots, P_k$ are norm-minimizing while all the other prime
ideals $P_{k+1},P_{k+2},\ldots$ are not. By Lemma~\ref{mini}(ii) the
(finite) set $\nmi$ of norm-minimizing ideals in the semigroup
$\cI_k$ generated by the $P_1,\ldots, P_k$ is in fact the finite set
of {\em all} norm-minimizing ideals of $R$. The projection
$\ve_{\nmi}:=\sum_{I\in\nmi}\ve_{IP_1\cdots P_k}$ corresponding to
the norm-minimizing ideals will be the key to our characterization
of ground states.

\begin{lemma}\label{ground-allnm} Let $\vp$ be a ground state of $\cT[R]$ and assume
  $n > k$ so that
$P_1,\ldots, P_k$ are norm-minimizing while $P_{k+1},P_{k+2},\ldots,
P_n$ are not. If $\ve_{\nmi}:=\sum_{I\in\nmi}\ve_{IP_1\cdots P_k}$,
then $\vp (\ve_{\nmi}\ve_{P_{k+1}P_{k+2}\cdots P_n})=1$.
\end{lemma}

\bproof Recall the minimal projections $\delta^x_{I,n}\in\cD_n$, for
$I\in\cI_n,\,x\in R/I$,   introduced in Section ~\ref{D}.
%We write $\delta_{I,n}$ for $\delta^0_{I,n}$.
Since $\ve_{IP_1P_2\ldots P_n}=\sum_{x\in R/I}\delta^x_{I,n}$ for
each $I\in \cI_n$, we have
\bgln\label{EE}\ve_{\nmi}\ve_{P_{k+1}P_{k+2}\cdots P_n}
=\sum_{I\in\nmi,\,x\in R/I}\delta^x_{I,n},\egln which is a
projection with finite support in $\Spec \bar\cD_n$. In view of
Lemma~\ref{mini}(i), the complement of the support is covered by the
supports of the $e_J^x$ with $J\in\cI_n\setminus \nmi$ and $ x\in
R/J$. By Lemma \ref{ground-notnm} we conclude that
\[
\vp\Big( 1 - \sum_{I\in \nmi, \, x\in R/I} \delta^x_{I,n}\Big) \leq
\sum_{J\in\cI_n\setminus \nmi,\, x\in R/J} \vp(e_J^x) =0,
\]
finishing the proof.\eproof

We will now consider $\cT[R]$ in its universal representation. Thus
let $S$ be the state space of $\cT[R]$ and let
$\pi_S=\bigoplus_{f\in S}\pi_f$ be its universal representation on
the Hilbert space $H_S=\bigoplus_{f\in S}H_f$. We will from now on
assume that $\cT[R]$ is represented via $\pi_S$ and we will omit the
$\pi_S$ from our notation.

If $\vp$ is a state of $\cT[R]$, we denote by $\tilde{\vp}$ its
unique normal extension to the von Neumann algebra $\cT[R]''$
generated by $\cT[R]$.

We write $\tilde{\delta}_I$, $\tilde{\delta}_I^x$, $\tilde{\ve}_I$
for the strong limits, as $n\to\infty$, of the monotonously
decreasing sequences of projections $\delta_{I,n}$, $\delta_{I,n}^x$
and $\ve_{IP_1P_2\ldots P_n}$, respectively (recall that
$\delta_{I,n} := \delta_{I,n}^0$).

In the representation $\mu$ used in section \ref{D}, the projection
$\tilde{\delta}_I^x$ is represented by the projection onto the
one-dimensional subspace of $\ell^2(R/I)$ corresponding to $x\in
R/I$. It is therefore non-zero.

We also consider the projection $E$ defined as the strong limit of
the sequence of projections $\ve_{\nmi}\ve_{P_{k+1}P_{k+2}\ldots
P_n}$. Equation \eqref{EE} immediately gives the formula

$$E=\sum_{I\in\nmi,\,x\in R/I}\tilde{\delta}_I^x.$$

\bprop\label{EEprop} A state $\vp$ of $\cT[R]$ is a ground state if
and only if $\tilde{\vp} (E)=1$.\eprop

\bproof If $\vp$ is a ground state, then $\tilde{\vp}(E) =1$ follows
immediately from Lemma~\ref{ground-allnm}
 because $\tilde{\vp}$ is normal.

If, conversely, $\tilde{\vp} (E)=1$, then
$\vp(w)=\tilde{\vp}(w)=\tilde{\vp}(EwE)$ for each $w\in\cT[R]$. In
order to show that condition (3) in Proposition~\ref{tent} is
satisfied, i.e. that $\tilde{\vp}(Es_b^*u^xs_as_a^*u^{-x}s_bE)=0$
whenever $N(a)>N(b)$, it suffices to show that
$\delta_I^ys_b^*u^xs_as_a^*u^{-x}s_b\delta_I^{y}=0$ for all $I\in
\nmi$ and $y\in R/I$, whenever $N(a)>N(b)$. This amounts to showing
that $\delta_{bI}^ys_as_a^*\delta_{bI}^{y}=0$ whenever $N(a)>N(b)$.
However, by equation \eqref{rem}, this last expression can be
non-zero only if $bI\subset aR$. This inclusion implies that
$I\subset \frac{a}{b}R\cap R$, i.e. that the ideal $\frac{a}{b}R\cap
R$ divides $I$. Since $I$ is norm-minimizing, $\frac{a}{b}R\cap R$
then has to be norm-minimizing, too, and $N(a)\leq N(b)$ by Lemma
\ref{mini}(iii).\eproof

\blemma\label{unit} Let $I, J\in \cI_n$ and let $a,b,a',b'\in R$
such that $aI=bJ$ and $a'I=b'J$. Then there is $g\in R^*$ such that
$s_{b'}^*s_{a'}=s_gs_b^*s_a$. The operators
$s_{a'}^*s_{b'}\tilde{\delta}_I$ and $s_b^*s_a\tilde{\delta}_I$ are
partial isometries with support $\tilde{\delta}_I$ and range
$\tilde{\delta}_J$. If $I,J,L$ are three ideals in $\cI_n$ and
$aI=bJ=cL$, then $s_c^*s_b\tilde{\delta}_J
s_b^*s_a\tilde{\delta}_I=s_c^*s_a\tilde{\delta}_I$\elemma

\bproof We have $(a/b)I=J=(a'/b')I$ whence $a'/b'=ga/b$ for some
$g\in R^*$. Thus $gab'=a'b$ and $s_gs_as_{b'}=s_{a'}s_b$.
Multiplying this from the left by $s_a^*s_{a'}^*$ gives the first
assertion (note that $s_g$ and $s_g^*$ commute with $s_a,s_{a'}$).
The second assertion then follows from Lemma \ref{inv}. Finally,
$s_c^*s_bs_b^*s_a\tilde{\delta} _I=s_c^*s_a\tilde{\delta}_I$ from
equation \eqref{rem} and the fact that $aI\subset bR$. \eproof

\bprop\label{EF} The corner $M=E\cT[R]E$ is a C*-algebra isomorphic
to

$$\bigoplus_{\gamma\in \Gamma} M_{k_\gamma
N(J_\gamma)}(C^*(J_\gamma)\rtimes R^*)$$

Here $\Gamma$ denotes the class group,
$k_\gamma=\abs{\nmi\cap\gamma}$ and $J_\gamma$ is any fixed ideal in
$\nmi\cap\gamma$ (they are all isomorphic).\eprop

\bproof We use the partition of $E$ as a sum of the projections
$\tilde{\delta}_I^x$, $I\in\nmi$, $x\in R/I$.

If $I_1,I_2\in\nmi$ are two ideals which are not in the same ideal
class and $w$ is an element of $\cT[R]$ of the form
$w=s_b^*e_Lu^ys_a$ with $a,b\in R^\times$, $y\in R$ then

$$\tilde{\delta}_{I_1}^{x_1}w\tilde{\delta}_{I_2}^{x_2}=
s_b^*\tilde{\delta}_{bI_1}^{bx_1}e_Lu^y\tilde{\delta}_{aI_2}^{ax_2}s_a=0$$

because $\tilde{\delta}_{L_1}^{t_1}\tilde{\delta}_{L_2}^{t_2}=0$ for
two different ideals $L_1,L_2$ independently of the choice of
$t_1,t_2$. Thus
$\tilde{\delta}_{I_1}^{x_1}\,\cT[R]\,\tilde{\delta}_{I_2}^{x_2}=0$.
If we write

$$E_\gamma=\sum_{I\in\nmi\cap\gamma,\,x\in R/I}\tilde{\delta}_I^x$$

then $E=\sum_{\gamma}E_\gamma$ and
$E_{\gamma_1}\cT[R]\,E_{\gamma_2}=0$ whenever $\gamma_1\neq
\gamma_2$.

If $I,J\in\nmi$ are two ideals in the same ideal class $\gamma$, we
can choose, according to Lemma \ref{unit}, a partial isometry
$c_{JI}$ of the form $c_{JI}=s_a^*s_b\tilde{\delta}_I$ with support
$\tilde{\delta}_I$ and range $\tilde{\delta}_J$. This element is
well determined up to multiplication by a unitary $s_g$, $g\in R^*$.
By fixing a reference ideal $J_\gamma$ in the class $\gamma$ and
choosing first the $c_{IJ_\gamma}$ and then putting $c_{LI}=
c_{LJ_\gamma}c_{IJ_\gamma}^*$, we may assume that the $c_{JI}$ have
the property that $c_{JI}=c_{IJ}^*$ and $c_{LJ}c_{JI}=c_{LI}$ for
$I,J,L\in\nmi\cap\gamma$ (i.e. they are matrix units). They generate
a matrix algebra isomorphic to $M_{k_\gamma}(\Cz)$. Setting

$$c_{IJ}^{xy}\, = \,u^xc_{IJ}u^{-y}$$

we obtain a system of matrix units for the larger index set
$\{(I,x)\, \big|\,I\in\nmi\cap\gamma\, , \,x\in R/I\}$. This system
generates a matrix algebra isomorphic to $M_{k_\gamma
N(J_\gamma)}(\Cz)$ (note that $N(I)=N(J_\gamma)$ for all
$I\in\nmi\cap\gamma$).

Consider again an element $w$ of $\cT[R]$ of the form
$w=s_b^*e_Lu^ys_a$ with $a,b\in R^\times$, $y\in R$. Then
$\tilde{\delta}_{J_\gamma}w\tilde{\delta}_{J_\gamma}$ is non-zero
only if $bJ_\gamma=aJ_\gamma$, $y\in aJ_\gamma$ and $L\supset
aJ_\gamma$. In that case we get

$$\tilde{\delta}_{J_\gamma}w\tilde{\delta}_{J_\gamma}=
u^{y/b}s_b^*s_a\tilde{\delta}_{J_\gamma}=u^{y/b}s_g\tilde
{\delta}_{J_\gamma}$$

for a suitable $g\in R^*$. This shows that
$\tilde{\delta}_{J_\gamma}\cT[R]\,\tilde{\delta}_{J_\gamma}$ is
isomorphic to the subalgebra $\mathfrak C$ of $\cT[R]$ generated by
the $s_g,\,g\in R^*$ and the $u^x,\,x\in J_\gamma$. On the other
hand the representation of $\cT[R]$ constructed in section \ref{>2}
shows that the surjective map $C^*(J_\gamma)\rtimes R^*\to \mathfrak
C$  from the crossed product is an isomorphism. Therefore
$\tilde{\delta}_{J_\gamma}\cT[R]\,\tilde{\delta}_{J_\gamma}$ is
isomorphic to the crossed product $C^*(J_\gamma)\rtimes R^*$.

Finally, the map that sends a matrix $(w_{I_1I_2}^{x_1x_2})$ in
$M_{k_\gamma N(J_\gamma)}(\mathfrak C)$ to

$$\sum c_{I_1J_\gamma}^{x_10}w_{I_1I_2}^{x_1x_2}c_{J_\gamma I_2}^{0x_2}$$

defines an isomorphism $M_{k_\gamma N(J_\gamma)}(\mathfrak C)\to
E_\gamma\cT[R]\,E_\gamma$.\eproof

\btheo The ground states of $\cT[R]$ are exactly the states of the
form $\vp (w)=\psi (EwE)$ where $\psi$ is an arbitrary state of
$E\cT[R]E\cong \bigoplus_\gamma M_{k_\gamma
N(J_\gamma)}(C^*(J_\gamma\rtimes R^*))$.\etheo

\bproof This is immediate from propositions \ref{EEprop} and
\ref{EF}.\eproof

\begin{appendix}
\section{Asymptotics for partial $\zeta$-functions}

As above let $R$ be the ring of algebraic integers in a number field
$K$. Also let $P_1,P_2,\ldots$ be an enumeration of the prime ideals
in $R$ such that $N (P_i) \le N (P_{i+1})$ for all $i \ge 1$ and let
$\cI_n$ be the semigroup generated by $P_1,P_2,\ldots ,P_n$. For
each $\gamma$ in the class group $\Gamma$ of $K$ and each
$0<\sigma\leq 1$ set

\bglnoz \zeta^{(n)}_\gamma(\sigma) =\sum_{I\in
\cI_n\cap\gamma}N(I)^{-\sigma}\eglnoz

Recall the statement of Theorem \ref{den}:\quad Let $0<\sigma\leq1$.
Then for any two ideal classes $\gamma_1,\gamma_2$ we have

$$\lim_{n\to\infty}\;\frac{\zeta^{(n)}_{\gamma_1}
(\sigma)}{\zeta^{(n)}_{\gamma_2} (\sigma)}\;=1.$$

\bproof[Proof of Theorem \ref{den}]

Let $\psi_\gamma :\Gamma\to \{0,1\}$ denote the characteristic
function of the one-point set $\{\gamma\}$. For every character
$\chi$ of the abelian group $\Gamma$ let $a_\gamma(\chi)=
|\Gamma|^{-1} \overline{\chi (\gamma)}$ so that

$$\psi_\gamma =\sum_{\chi\in\hat{\Gamma}}a_\gamma(\chi)\chi$$

In the following we also consider $\chi$ and $\psi_\gamma$ as
functions on the set of non-zero integral ideals. We have

\bgln\label{one}\zeta_\gamma^{(n)}(\sigma)=\sum_{I\in\cI_n}\psi_{\gamma}
(I)N(I)^{-\sigma}= \sum_{\chi\in\hat{\Gamma}}\Big
(a_\gamma(\chi)\sum_{I\in\cI_n}\chi
(I)N(I)^{-\sigma}\Big)\\
\nonumber=\sum_{\chi\in\hat{\Gamma}}\Big(a_\gamma(\chi)
\prod_{i=1}^n \Big(1-\chi (P_i)N(P_i)^{-\sigma}\Big)^{-1}\Big)\egln

In order to study the asymptotics of $\prod_{i=1}^n \left(1-\chi
(P_i)N(P_i)^{-\sigma}\right)^{-1}$ for $n\to\infty$ we consider

\bgln\label{two} f_n(\chi,\sigma)=\log \prod_{i=1}^n \left(1-\chi
(P_i)N(P_i)^{-\sigma}\right)^{-1}\\ \nonumber :=\sum_{\nu
=1}^\infty\frac{1}{\nu}\sum_{i=1}^n \chi
(P_i^\nu)N(P_i)^{-\nu\sigma}\egln

Up to finitely many terms the first sum is bounded by a constant
which is independent of $n$:
\begin{eqnarray*}
 \Big| \sum_{\nu > 1/\sigma} \frac{1}{\nu} \sum^n_{i=1} \chi (P^{\nu}_i) N (P_i)^{-\nu\sigma}\Big| & \le & \sum_{\nu > 1/\sigma} \frac{1}{\nu} \sum^{\infty}_{i=1} N (P_i)^{-\nu \sigma} \\
& = & \sum^{\infty}_{i=1} N(P_i)^{-\sigma [1/\sigma]} \sum^{\infty}_{\nu=1} \frac{N (P_i)^{-\nu\sigma}}{\nu + [1/ \sigma]} \\
& \le & \sum^{\infty}_{i=1} N (P_i)^{-\sigma [1/\sigma]} \frac{N(P_i)^{-\sigma}}{1-N (P_i)^{-\sigma}} \\
& \le & \frac{1}{1-2^{-\sigma}} \sum^{\infty}_{i=1} N (P_i)^{-\sigma (1 + [1/\sigma])} \\
& < & \frac{1}{1-2^{-\sigma}} \zeta_K (\sigma (1 + [1/\sigma]) <
\infty \; .
\end{eqnarray*}

Therefore

\bgln\label{three} f_n (\chi , \sigma) =\sum_{1\leq
\nu\leq1/\sigma}\frac{1}{\nu}\sum_{i=1}^n \chi
(P_i^\nu)N(P_i)^{-\nu\sigma}\,+\,O(1)\egln

where the $O$-constant depends on $\sigma$ but not on $n$ or $\chi$.

Let us now fix some $1 \le \nu \le 1/\sigma$. The values of $\chi$
are $h$-th roots of unity where $h=\abs{\Gamma}$ is the class
number. We get

\bgln \label{four} \sum^n_{i=1} \chi (P^{\nu}_i) N(P_i)^{-\nu
\sigma} = \sum_{\zeta^h=1} \zeta^{\nu} \sum_{\gamma \in \chi^{-1}
(\zeta)} \omega^{(\nu\sigma)}_{\gamma} (n) \; . \egln

Here for $\kappa \in \mathbb{R} , \gamma \in \Gamma$ and $n \ge 1$
we have set:
\[
 \omega^{(\kappa)}_{\gamma} (n) = \sum^n_{i=1 \atop P_i \in \gamma} N (P_i)^{-\kappa} \; .
\]

\blemma \label{deningerapp1} Fix some $0 \le \kappa \le 1$  and
write $\omega_{\gamma} (n) = \omega^{(\kappa)}_{\gamma} (n)$. Set
\[
 \omega (n) = \frac{1}{h} \sum^n_{i=1} N (P_i)^{-\kappa} \; .
\]
Then $\omega (n) \to \infty$ as $n \to \infty$ and for arbitrary
$\gamma \in \Gamma$ we have $\lim_{n\to\infty} \frac{\omega_{\gamma}
(n)}{\omega (n)} = 1$. \elemma

The proof of the lemma is given below. For $1 \le \nu \le 1/\sigma$
we have $0 < \kappa = \nu\sigma \le 1$. Using \eqref{four} and the
lemma, we get for $n \to \infty$:
\[
 \frac{1}{\omega (n)} \sum^n_{i=1} \chi (P^{\nu}_i) N (P_i)^{-\nu\sigma} \to \sum_{\zeta^h=1} \zeta^{\nu} |\chi^{-1} (\zeta)| \; .
\]

We have the identities \bglnoz \sum_{\zeta^h =1}\zeta^\nu
|\chi^{-1}(\zeta)|= |\Ker(\chi)|\sum_{\zeta \in \Img \chi}\zeta^\nu
=\left\{\begin{array}{l} h \quad\mbox{if}\;\abs{\Img
\chi}\,\mid\,\nu
\\[3pt]
0\quad\mbox{if}\;\abs{\Img \chi}\,\nmid
\,\nu\end{array}\right.\eglnoz

Therefore, using (\ref{three}) we get

\bgln\label{seven} \lim_{n\to \infty} \;
\frac{1}{\omega(n)}f_n(\chi,\sigma) = \alpha (\chi) :=
h\sum_{1\leq\nu\leq 1/\sigma,\,\abs{\Img \chi}\big|\nu}\frac{1}{\nu}
\geq 0 \egln

Note that if $\chi$ is not the trivial character $\textbf{1}$, then
$\alpha (\chi)<\alpha (\textbf{1})$.

Let \bgln \label{eight} L_n(\chi,\sigma)\defeq \prod_{i=1}^n
\left(1-\chi (P_i)N(P_i)^{-\sigma}\right)^{-1}=\exp
f_n(\chi,\sigma)\egln

From (\ref{one}) we get

$$\zeta^{(n)}_\gamma(\sigma) = \sum_\chi
a_\gamma(\chi)L_n(\chi,\sigma)$$

Because of \eqref{seven} and \eqref{eight} one knows that for
$n\to\infty$

$$0\,<\,L_n(\textbf{1},\sigma) = \prod_{i=1}^n \left(1-N(P_i)^{-\sigma}
\right)^{-1}\lori\infty$$

Also

$$\Big|\frac{L_n(\chi,\sigma)}{L_n(\textbf{1},\sigma)}\Big|\,=\,\exp
\Rea \left(f_n(\chi,\sigma)-f_n(\textbf{1},\sigma)\right) $$

Now assume that $\chi\neq \textbf{1}$. Since $\omega (n)\to\infty$
and

$$\lim_{n\to\infty} \frac{1}{\omega
(n)}\left(f_n(\chi,\sigma)-f_n(\mathbf{1},\sigma)\right) =
\alpha(\chi)-\alpha (\mathbf{1})\,<\,0$$

by \eqref{seven}, we find that

$$\lim_{n\to\infty}\Rea \left(f_n(\chi,\sigma)-f_n(\mathbf{1},\sigma)\right)
=\,-\infty$$

and thus

$$\lim_{n\to\infty}\frac{L_n(\chi,\sigma)}{L_n(\textbf{1},\sigma)}=\,0\qquad
{\mathrm for}\,\chi\neq \textbf{1}$$

This gives:

$$\lim_{n\to\infty}\frac{\zeta^{(n)}_\gamma(\sigma)}
{L_n(\mathbf{1},\sigma)}=a_\gamma (\mathbf{1}) =\frac{1}{h}$$

and hence

$$\lim_{n\to\infty}\frac{\zeta^{(n)}_\gamma(\sigma)}
{\zeta^{(n)}_\eta(\sigma)}=\,1$$

for any two ideal classes $\gamma$ and $\eta$.

It remains to prove lemma \ref{deningerapp1}. For this we need a
version of the prime number theorem for prime ideals in a given
ideal class with a simple remainder term. For $x \ge 0$ let $\pi_K
(\gamma , x)$ denote the number of prime ideals $P$ in $\gamma$ with
$N (P) \le x$. Using the relation
\[
\mathrm{li}\, (x) = \frac{x}{\log x} + O \Big( \frac{x}{(\log x)^2}
\Big) \quad \text{for} \; x \to \infty
\]
 the corollary after lemma 7.6 in chap. 7, \S\,2 of \cite{Narkiewicz} implies the following asymptotics:
\begin{equation} \label{nine}
\pi_K (\gamma , x) = \frac{1}{h} \frac{x}{\log x} + O \Big(
\frac{x}{(\log x)^2} \Big)
\end{equation}
For $x \ge 0$ and $\kappa \le 1$ let us write:
\[
\Omega_{\gamma} (x) = \Omega^{(\kappa)}_{\gamma} (x) = \sum_{N (P)
\le x \atop P \in \gamma} N (P)^{-\kappa}
\]
and
\[
\Omega (x) = \Omega^{(\kappa)} (x) = \frac{1}{h} \sum_{N (P) \le x}
N (P)^{-\kappa} \; .
\]
We now use the following version of summation by parts: Consider a
function $f$ on the integers $\nu \ge 1$ and a $C^1$-function $g$ on
$[1,\infty)$. For $x \ge 1$ we set $M_f (x) = \sum_{\nu \le x} f
(\nu)$. Then we have
\[
\sum_{\nu \le x} f (\nu) g (\nu) = M_f (x) g (x) - \int^x_1 M_f (t)
g' (t) \, dt \; .
\]
Setting $f (\nu) = \big| \{ P \tei P \in \gamma$ and $N (P) = \nu
\}\big|$ and $g (x) = x^{-\kappa}$ we have
\[
\Omega_{\gamma} (x) = \sum_{\nu \le x} f (\nu) g (\nu) \quad
\text{and} \quad M_f (x) = \pi_K (\gamma , x) \; .
\]
Hence using \eqref{nine} we get for $x \to \infty$:
\begin{eqnarray*}
\Omega_{\gamma} (x) & = & \pi_K (\gamma , x) x^{-\kappa} + \kappa
\int^x_2
\pi_K (\gamma , t) t^{-\kappa} \frac{dt}{t}\\
& = & \frac{1}{h} \frac{x^{1-\kappa}}{\log x} + \frac{\kappa}{h}
\int^x_2 \frac{t^{-\kappa}}{\log t} \, dt + O \Big(
\frac{x^{1-\kappa}}{(\log x)^2} \Big) + O \Big( \int^x_2
\frac{t^{-\kappa}}{(\log t)^2} \, dt \Big) \; .
\end{eqnarray*}
For $\kappa < 1$ we have:
\begin{eqnarray*}
\int^x_e \frac{t^{-\kappa}}{(\log t)^2} \, dt & = &
\int^{\sqrt{x}}_e \frac{t^{-\kappa}}{(\log t)^2} \, dt +
\int^x_{\sqrt{x}}
\frac{t^{-\kappa}}{(\log t)^2} \, dt \\
& \le & \int^{\sqrt{x}}_e t^{-\kappa} \, dt + \frac{1}{(\log
\sqrt{x})^2}
\int^x_{\sqrt{x}} t^{-\kappa} \, dt \\
& = & O \Big( \frac{x^{1-\kappa}}{(\log x)^2} \Big)\; .
\end{eqnarray*}
Hence we get for $\kappa < 1$:
\begin{equation} \label{ten}
\Omega_{\gamma} (x) = \frac{1}{h} \frac{x^{1-\kappa}}{\log x} +
\frac{\kappa}{h} \int^x_2 \frac{t^{-\kappa}}{\log t} \, dt + O \Big(
\frac{x^{1-\kappa}}{(\log x)^2} \Big) \; .
\end{equation}
For the case $\kappa = 1$ note that
\[
\int^x_2 \frac{t^{-1}}{(\log t)^2} \, dt = \frac{1}{\log 2} -
\frac{1}{\log x} = O (1)
\]
and
\[
\int^x_2 \frac{t^{-1}}{\log t} \, dt = \log \log x + O (1) \; .
\]
Thus for $\kappa = 1$ we get
\begin{equation} \label{eleven}
\Omega_{\gamma} (x) = \frac{1}{h} \log \log x + O (1) \; .
\end{equation}
Relations \eqref{ten} and \eqref{eleven} also hold for $\Omega (x)$
instead of $\Omega_{\gamma} (x)$ since the right hand sides do not
depend on $\gamma$ and $\Omega (x) = h^{-1} \sum_{\gamma \in \Gamma}
\Omega_{\gamma} (x)$. It follows that for $\kappa \le 1$ we have
$\Omega_{\gamma} (x) \sim \Omega (x)$. It remains to show that for
$n \to \infty$ we have $\omega_{\gamma} (n) \sim \omega (n)$ as
well. For a given prime number $p$ there are at most $(K :
\mathbb{Q})$ different prime ideals $P$ in $R$ with $P \tei p$. It
follows that for every $\nu \ge 1$ the equation $N (P) = \nu$ has at
most $(K : \mathbb{Q})$ solutions in primes $P$ of $R$. Since $N
(P_i) \le N (P_{i+1})$ for all $i$ we therefore get:
\begin{eqnarray*}
\omega_{\gamma} (n) & = & \Omega_{\gamma} (N (P_n)) + O (N
(P_n)^{-\kappa})
\\
& = & \Omega_{\gamma} (N (P_n)) + O (1) \quad \text{since} \;
\kappa\ge 0
\end{eqnarray*}
and analogously
\[
\omega (n) = \Omega (N (P_n)) + O (1) \; .
\]
This implies the result.\eproof

\section{List of notations}
For $u^x,s_a,e_I$ see Definition 2.1. The projections $f_I,\ve_I$
are introduced before Lemma 2.4. The commutative subalgebra
$\bar{\cD}$ is introduced at the beginning of section 4.
$\cI_n,\cD_n,\bar{\cD}_n$ are introduced after Lemma 4.5. The
representation $\mu$ of $\cT[R]$ is defined before Lemma 4.6. The
minimal projections $\delta^x_{I,n}$ in $\cD_n$ are introduced in
Lemma 4.7. $e_I^x$ is defined after 4.8. For $Y_R$ see Remark 4.10.
The notation $\cT$ is introduced after Corollary 4.14. The
automorphism $\sigma_t$ is defined at the beginning of section 6.
$R^*$ denotes the group of units (invertible elements) in $R$.
\end{appendix}

\end{document}